\documentclass[reqno]{amsart}

\usepackage{times}
\usepackage{mathabx}
\usepackage{amsmath}
\usepackage{amsthm}
\usepackage{amsfonts}

\newtheorem{thm}{Theorem}[section]
\newtheorem{lem}[thm]{Lemma}
\newtheorem{cor}[thm]{Corollary}

\newtheorem{rem}[thm]{Remark}

\DeclareMathAlphabet{\mathpzc}{OT1}{pzc}{m}{it}

\numberwithin{equation}{section}

\newcommand{\Wqb}{W_{q,D}}
\newcommand{\Wqq}{\Wqb^{2-2/q}}
\newcommand{\Wqqp}{\Wqb^{2-2/q,+}}
\newcommand{\R}{\mathbb{R}}

\newcommand{\N}{\mathbb{N}}

\newcommand{\Wq}{\mathbb{W}_q}
\newcommand{\Wqd}{\dot{\mathbb{W}}_q^+}
\newcommand{\Lq}{\mathbb{L}_q}

\newcommand{\ml}{\mathcal{L}}
\newcommand{\mk}{\mathcal{K}}

\newcommand{\Om}{\Omega}

\newcommand{\ve}{\varepsilon}
\newcommand{\rd}{\mathrm{d}}

\newcommand{\bqn}{\begin{equation}}
\newcommand{\eqn}{\end{equation}}
\newcommand{\bqnn}{\begin{equation*}}
\newcommand{\eqnn}{\end{equation*}}
\newcommand{\bear}{\begin{eqnarray}} 
\newcommand{\eear}{\end{eqnarray}} 
\newcommand{\bean}{\begin{eqnarray*}} 
\newcommand{\eean}{\end{eqnarray*}} 
\newcommand{\bs}{\begin{split}}
\newcommand{\es}{\end{split}}

\newcommand{\dhr}{\mathrel{\lhook\joinrel\relbar\kern-.8ex\joinrel\lhook\joinrel\rightarrow}}

\title[Steady-State Equations of Cooperative or Competing Systems]
{On Nonlocal Parabolic Steady-State Equations of Cooperative or Competing Systems}

\author[Ch. Walker]{Christoph Walker}

\address{Leibniz Universit\"at Hannover, Institut f\"ur Angewandte Mathematik, Welfengarten 1, D--30167 Hannover, Germany.}
\email{walker@ifam.uni-hannover.de}

\begin{document}

\begin{abstract}
Some systems of parabolic equations with nonlocal initial conditions are studied. The systems arise when considering steady-state solutions to diffusive age-structured cooperative or competing species. 
Local and global bifurcation techniques are employed to derive a detailed description of the structure of positive coexistence solutions.
\end{abstract}


\maketitle

\section{Introduction and Main Results}\label{sec 0}

In this paper we characterize the structure of positive solutions to certain systems of coupled parabolic equations with nonlocal initial conditions. Such systems arise as steady-state equations of two age-structured diffusive populations which inhabit the same spatial region. The interaction between the two species is either of cooperative, competing, or predator-prey type leading to different structures of positive solutions. Denoting the density of the two species by $u=u(a,x)\ge 0$ and $v=v(a,x)\ge 0$ with $a\in (0,a_m)$ and $x\in\Om\subset\R^n$ referring to age and spatial position, respectively, the models we shall focus on are of the form
\begin{align}
\partial_a u-\Delta_D u&=-\alpha_1u^2\pm\alpha_2 vu\ ,\quad  a\in (0,a_m)\ ,\quad x\in\Om\ ,\label{1}\\
\partial_a v-\Delta_D v&=-\beta_1 v^2\pm\beta_2 u v\ ,\,\quad  a\in (0,a_m)\ ,\quad x\in\Om\ ,\label{2}
\end{align}
subject to the nonlocal initial conditions
\begin{align}
u(0,x)&=\eta\int_0^{a_m}b_1(a)\, u(a,x)\, \rd a\ , \quad   x\in\Om\ ,\label{3}\\
v(0,x)&=\xi\int_0^{a_m}b_2(a)\, v(a,x)\, \rd a\ ,  \quad x\in\Om\ .\label{4}
\end{align}
The equations are the steady-state equations of the corresponding time-dependent age-structured equations with diffusion. We refer to \cite{WebbSpringer} for a recent survey on the formidable literature about age-structured population models.  

The operator $-\Delta_D$ in \eqref{1}, \eqref{2} stands for the negative Laplacian on $\Om$ with subscript $D$ indicating that Dirichlet conditions
$$
 u(a,x)=v(a,x)=0\ ,\quad a\in (0,a_m)\ ,\quad x\in\partial\Om\ ,
$$
are imposed on the smooth boundary $\partial\Om$ of the bounded domain $\Om$. Normalization to $1$ of the diffusion coefficients in \eqref{1}, \eqref{2} is a purely notational simplification. The number $a_m>0$ denotes the maximal age of the species. Equations \eqref{3}, \eqref{4} represent the age-boundary conditions reflecting that individuals with age zero are those created when a mother individual of any age $a\in (0,a_m)$ gives birth with rates $\eta b_1(a)$ or $\xi b_2(a)$. The functions $b_j=b_j(a)\ge 0$ describe the profiles of the fertility rates while the parameters $\eta, \xi>0$ measure their intensity without affecting the structure of the birth profiles. For easier statements of the results we assume the birth profiles
\bqn\label{5}
b_j\in L_\infty^+((0,a_m))\ \text{with}\ b_j(a)>0\ \text{for $a$ near $a_m$}\ ,\quad j=1,2\ ,
\eqn
are normalized such that
\bqn\label{6}
\int_0^{a_m}b_j(a)e^{-\lambda_1 a}\,\rd a=1\ ,\quad j=1,2\ ,
\eqn
where $\lambda_1>0$ denotes the principal eigenvalue of $-\Delta_D$ on $\Om$. 

Assuming $\alpha_1,\alpha_2,\beta_1,\beta_2>0$, the form of the interaction between the two species is determined by the signs on the right hand side of equations \eqref{1}, \eqref{2}. Replacing $\pm$ by a positive sign $+$ in both of the equations \eqref{1} and \eqref{2}) corresponds to a system (see \eqref{1coo} , \eqref{2coo} below) where the two species are {\it cooperative}, while the case with $\pm$ replaced by negative signs~$-$ in each equation \eqref{1} and \eqref{2} (see \eqref{1comp}, \eqref{2comp} below) reflects a {\it competition} of the species. The case with mixed signs, e.g. a negative sign~$-$ in \eqref{1} instead of $\pm$ and a positive sign~$+$ in \eqref{2} describes a {\it predator-prey}-system (see \eqref{1pp}, \eqref{2pp} below) with a prey density $u$ and a predator density~$v$. 

This last case of a predator-prey-system was studied in \cite{WalkerJRAM} and local and global bifurcation phenomena of positive nontrivial solutions with respect to the parameters $\eta$ and $\xi$ were obtained.
In the present paper, we shall derive global bifurcation results for the cooperative and the competition case. Depending on $\eta$ and $\xi$ we shall give a rather complete description of positive {\it coexistence solutions}, that is, of solutions $(u,v)$ with both components $u$ and $v$ positive and nontrivial. Moreover, we shall also improve the local bifurcation result \cite[Thm.2.4]{WalkerJRAM} to a global one.

We like to point out that variants of the elliptic counterparts to \eqref{1}-\eqref{2} when age-structure is neglected from the outset have been intensively studied in the past, e.g. see \cite{BlatBrown1,BlatBrown2,CantrellCosner,CasalEilbeckLG,CosnerLazer,DancerTAMS84,DancerJDE85,DancerLGOrtega,Leung,LopezGomezPardo,LopezGomezChapman,LopezGomezMolinaJDE06,LopezGomezJDE09,SchiaffinoTesei,ZhouPao}. Concerning equations for a single specie, e.g. variants of \eqref{1} subject to \eqref{3}, we refer to \cite{DelgadoEtAl,DelgadoEtAl2,WalkerSIMA,WalkerJDE,WalkerArchMath,WalkerAMPA}.\\

To state our results for the present situation, we shall keep $\xi$ fixed and regard $\eta$ as bifurcation parameter in the following. We thus write $(\eta,u,v)$ for solutions to \eqref{1}-\eqref{4} with $u$, $v$ belonging to the positive cone $\Wq^+$ of
$$
\Wq:=L_q((0,a_m),\Wqb^2(\Om))\cap W_q^1((0,a_m),L_q(\Om))
$$
for $q>n+2$ fixed, but remark that all our solutions will have smooth components $u$, $v$ with respect to both $a\in J$ and $x\in\Om$. We say that a continuum $\mathfrak{C}$ (i.e. a closed and connected set) in $\R^+\times\Wq^+\times\Wq^+$ of solutions $(\eta,u,v)$ to \eqref{1}-\eqref{4} is {\it unbounded with respect to $\eta$}, provided the $\eta$-projection of $\mathfrak{C}$ contains an interval of the form $(\eta_0,\infty)$ with $\eta_0\in\R^+$, and we say that $\mathfrak{C}$ is {\it unbounded with respect to the $u$-component in $\Wq$} provided there is a sequence $((\eta_j,u_j,v_j))_{j\in\N}$ in $\mathfrak{C}$ with $\|u_j\|_{\Wq}\rightarrow \infty$ as $j\rightarrow\infty$. An analogous terminology shall be used if $\mathfrak{C}$ is unbounded with respect to the $v$-component.\\

Clearly, problem \eqref{1}-\eqref{4} decouples when taking either $u$ or $v$ to be zero. Noticing that Theorem~\ref{A1} from the appendix provides for each $\eta>1$ a unique solution $u_\eta\in\Wq^+\setminus\{0\}$ to
\bqn\label{Aee1}
\partial_a u-\Delta_D u=-\alpha_1u^2\ ,\quad u(0,\cdot)=\eta\int_0^{a_m}b_1(a)\, u(a,\cdot)\,\rd a\ ,
\eqn
and similarly for each $\xi>1$ a unique solution $v_\xi\in\Wq^+\setminus\{0\}$ to
\bqn\label{Aee2}
\partial_a v-\Delta_D v=-\beta_1 v^2\ ,\quad v(0,\cdot)=\xi\int_0^{a_m}b_2(a)\, v(a,\cdot)\,\rd a\ ,
\eqn
there is, independent of what the signs $\pm$ in \eqref{1}, \eqref{2} are, for any $\xi\ge 0$ the trivial branch
$$
\mathfrak{B}_0:=\{(\eta,0,0)\,;\, \eta\ge 0\}
$$
and the semi-trivial branch
\bqn\label{k1}
\mathfrak{B}_1:=\{(\eta,u_\eta,0)\,;\, \eta> 1\}\subset\R^+\times(\Wq^+\setminus\{0\})\times\Wq^+
\eqn
of solutions. For $\xi>1$, an additional semi-trivial branch
\bqn\label{k2}
\mathfrak{B}_2:=\{(\eta,0,v_\xi)\,;\, \eta\ge 0\}\subset\R^+\times\Wq^+\times(\Wq^+\setminus\{0\})\ 
\eqn
exists. Depending on the signs $\pm$ in \eqref{1}, \eqref{2} we shall establish further local and global bifurcation of coexistence solutions from these semi-trivial branches.

\subsection{Cooperative Systems}\label{sec 1b}

We first consider the cooperative case
\begin{align}
\partial_a u-\Delta_D u&=-\alpha_1u^2+\alpha_2 vu\ ,\quad  a\in (0,a_m)\ ,\quad x\in\Om\ ,\label{1coo}\\
\partial_a v-\Delta_D v&=-\beta_1 v^2+\beta_2 u v\ ,\,\quad  a\in (0,a_m)\ ,\quad x\in\Om\ ,\label{2coo}
\end{align}
subject to the nonlocal initial conditions
\begin{align*}
u(0,x)=\eta\int_0^{a_m}b_1(a)\, u(a,x)\, \rd a\ , \quad   
v(0,x)=\xi\int_0^{a_m}b_2(a)\, v(a,x)\, \rd a\ ,  \qquad x\in\Om\ .
\end{align*} 
Recall that $\mathfrak{B}_1$ is the only semi-trivial branch of solutions to \eqref{1coo}-\eqref{2coo} if $\xi<1$.
Thus, if $\xi<1$ we shall derive bifurcation with respect to the parameter $\eta$ from the semi-trivial branch $\mathfrak{B}_1$ consisting of solutions of the form $(\eta,u_\eta,0)$. The case $\xi<1$ is therefore more involved than the case $\xi>1$ where we shall establish a bifurcation from the semi-trivial branch $\mathfrak{B}_2$ consisting of solutions of the form $(\eta,0,v_\xi)$.

\begin{thm}\label{T2}
There is $\nu\in [0,1)$ with the property that for each $\xi\in (\nu,1)$ there exists $\eta_0:=\eta_0(\xi)>1$ such that $(\eta_0,u_{\eta_0},0)\in\mathfrak{B}_1$ is a bifurcation point. There is an unbounded continuum $\mathfrak{C}_1$ of coexistence solutions $(\eta,u,v)$ in $\R^+\times(\Wq^+\setminus\{0\})\times (\Wq^+\setminus\{0\})$ to \eqref{1coo}-\eqref{2coo} subject to \eqref{3}-\eqref{4} emanating from $(\eta_0,u_{\eta_0},0)$. Near the branch $\mathfrak{B}_1$, the continuum $\mathfrak{C}_1$ is a continuous curve. There is no other bifurcation point on $\mathfrak{B}_1$ to positive coexistence solutions.
\end{thm}

The precise values of $\nu$ and of $\eta_0(\xi)>1$ are related to spectral radii of some operators and are given in \eqref{nu} and \eqref{xi1}, respectively. Actually, we conjecture that $\nu=0$, see Remark~\ref{R32}.\\

If $\xi>1$, then a global continuum of coexistence solutions emanates from the branch $\mathfrak{B}_2$:

\begin{thm}\label{T1}
Given $\xi>1$, there is $\eta_1:=\eta_1(\xi)\in (0,1)$ such that $(\eta_1,0,v_\xi)\in\mathfrak{B}_2$ is a bifurcation point. An unbounded continuum $\mathfrak{C}_2$ of coexistence solutions $(\eta,u,v)$ in $\R^+\times(\Wq^+\setminus\{0\})\times (\Wq^+\setminus\{0\})$ to \eqref{1coo}-\eqref{2coo} subject to \eqref{3}-\eqref{4} emanates from $\big(\eta_1,0,v_\xi\big)$. Near the branch $\mathfrak{B}_2$, the continuum $\mathfrak{C}_2$ is a continuous curve. There is no other bifurcation point on $\mathfrak{B}_2$ or on $\mathfrak{B}_1$ to positive coexistence solutions.
\end{thm}

The precise value of $\eta_1(\xi)$ is given in \eqref{17}. 
We can give a more specific characterization of the global nature of the the continua $\mathfrak{C}_1$ and $\mathfrak{C}_2$:

\begin{cor}\label{C0}
The global continua $\mathfrak{C}_1$ and $\mathfrak{C}_2$ provided by Theorem~\ref{T2} and Theorem~\ref{T1}, respectively, are unbounded with respect to both the parameter $\eta$ and the $u$-component in $\Wq$, or with respect to the $v$-component in $\Wq$. If 
\bqn\label{B}
b_2\in L_1((0,a_m),(1-e^{-sa})^{-1}\rd a)
\eqn
for some $s>0$, then they are unbounded with respect to the $u$-component in $\Wq$. 
\end{cor}

\subsection{Competing Systems}\label{sec 1c}

Next we consider the case of two competing species:
\begin{align}
\partial_a u-\Delta_D u&=-\alpha_1u^2-\alpha_2 vu\ ,\quad  a\in (0,a_m)\ ,\quad x\in\Om\ ,\label{1comp}\\
\partial_a v-\Delta_D v&=-\beta_1 v^2-\beta_2 u v\ ,\,\quad  a\in (0,a_m)\ ,\quad x\in\Om\ ,\label{2comp}
\end{align}
subject to the initial conditions
\begin{align*}
u(0,x)=\eta\int_0^{a_m}b_1(a)\, u(a,x)\, \rd a\ , \quad   
v(0,x)=\xi\int_0^{a_m}b_2(a)\, v(a,x)\, \rd a\ ,  \qquad x\in\Om\ .
\end{align*} 
The following theorem characterizes the competition coexistence solutions.

\begin{thm}\label{T3}
If $\xi\le 1$, then there is no coexistence solution $(\eta,u,v)\in\R^+\times(\Wq^+\setminus\{0\})\times(\Wq^+\setminus\{0\})$ to \eqref{1comp}-\eqref{2comp} subject to \eqref{3}-\eqref{4}. If $\xi>1$, then there is $\eta_2:=\eta_2(\xi)>1$ such that $\big(\eta_2,0,v_{\xi}\big)\in\mathfrak{B}_2$ is a bifurcation point. A continuum $\mathfrak{C}_3$ of positive coexistence solutions in $\R^+\times(\Wq^+\setminus\{0\})\times (\Wq^+\setminus\{0\})$ emanates from $\big(\eta_2,0,v_{\xi}\big)$ satisfying the alternative
\begin{itemize}
\item[(a)] $\mathfrak{C}_3$ joins $\mathfrak{B}_2$ with $\mathfrak{B}_1$, or
\item[(b)] $\mathfrak{C}_3$ is unbounded with respect to the parameter $\eta$,
\end{itemize}
and near the bifurcation point $\big(\eta_2,0,v_{\xi}\big)$, the continuum $\mathfrak{C}_3$ is a continuous curve.
There exists some $N>1$ such that alternative (a) occurs for each $\xi\in (1,N)$. Moreover, if
\bqn\label{dominated}
\beta_2\ge \alpha_1\ ,\quad \beta_1\ge\alpha_2\ ,\quad b_1\ge b_2\ \text{on}\ (0,a_m)\ ,
\eqn
then the $\eta$-projection of $\mathfrak{C}_3$ is contained in the interval $(1,\xi]$; in particular, alternative (a) occurs for each~$\xi>1$.
\end{thm}

The value of $\eta_2(\xi)$ as well as the value of $\eta_3:=\eta_3(\xi)$ corresponding to the point $\big(\eta_3,u_{\eta_3},0\big)\in\mathfrak{B}_1$ where $\mathfrak{C}_3$ joins up with $\mathfrak{B}_1$ if alternative (a) occurs, are determined exactly, see \eqref{xi2} and \eqref{xi3}. Actually, we conjecture $N=\infty$ even if \eqref{dominated} does not hold, see Remark~\ref{R42}. Observe that \eqref{dominated} implies a biological advantage of the specie with density $u$ due to a higher birth but lower death rate.

\subsection{Predator-Prey-Systems}\label{sec 1d}

The case of a predator-prey-system,
\begin{align}
\partial_a u-\Delta_D u&=-\alpha_1u^2-\alpha_2 vu\ ,\quad  a\in (0,a_m)\ ,\quad x\in\Om\ ,\label{1pp}\\
\partial_a v-\Delta_D v&=-\beta_1 v^2+\beta_2 u v\ ,\,\quad  a\in (0,a_m)\ ,\quad x\in\Om\ ,\label{2pp}
\end{align}
subject to the initial conditions
\begin{align*}
u(0,x)=\eta\int_0^{a_m}b_1(a)\, u(a,x)\, \rd a\ , \quad   
v(0,x)=\xi\int_0^{a_m}b_2(a)\, v(a,x)\, \rd a\ ,  \qquad x\in\Om\ ,
\end{align*} 
was  studied in detail in \cite{WalkerJRAM}. A quite complete description of the structure of positive solutions was provided when $\xi$ is regarded as bifurcation parameter and $\eta>0$ is kept fixed \cite[Thm.2.2]{WalkerJRAM} or when $\eta$ is regarded as bifurcation parameter and $\xi>1$ is kept fixed [Thm.2.3]. However, for the case $\xi<1$ being fixed with parameter $\eta$, a {\it local} bifurcation and thus a merely partial result was obtained in \cite[Thm.2.4]{WalkerJRAM}. More precisely, provided that $\xi\in (\delta,1)$ for some suitable \mbox{$\delta\in [0,1)$}, it was shown in \cite[Thm.2.4]{WalkerJRAM} that there are $\ve_0>0$ and a unique point $(\eta_4,u_{\eta_4},0)$ with $\eta_4:=\eta_4(\xi)>1$ on the semi-trivial branch $\mathfrak{B}_1$
such that a local continuous curve
$$
\mathfrak{C}_4:=\{(\eta(\ve),u(\ve),v(\ve))\,;\,0<\ve<\ve_0 \}
\subset \R^+\times (\Wq^+\setminus\{0\})\times(\Wq^+\setminus\{0\})
$$
of positive coexistence solutions bifurcates to the right from $(\eta_4,u_{\eta_4},0)$. Actually, this result can be improved:

\begin{thm}\label{T4}
The branch $\mathfrak{C}_4$ extends to an unbounded continuum in $\R^+\times (\Wq^+\setminus\{0\})\times(\Wq^+\setminus\{0\})$ of positive coexistence solutions to \eqref{1pp}-\eqref{2pp} subject to \eqref{3}-\eqref{4}. If \eqref{B} holds, then $\mathfrak{C}_4$ is unbounded with respect to the parameter $\eta$.
\end{thm}

The proof of this theorem is along the lines of the one of Theorem~\ref{T2} with only minor modifications necessary. We shall thus omit it here. \\

The outline of the present paper is as follows: In the next section, Section~\ref{sec 2}, notations and some preliminary results are introduced. In Section~\ref{sec 2b2}  a detailed proof of Theorem~\ref{T2} is provided so that the proof of Theorem~\ref{T1} in Section~\ref{sec 2b1} is basically a straightforward modification thereof and can thus be kept short. Section~\ref{sec 2c} is dedicated to the proof of Theorem~\ref{T3}.
Appendix~A contains some results regarding the semi-trivial branches induced by \eqref{Aee1}, \eqref{Aee2} which are of importance for the proofs in Sections~\ref{sec 2b2}-\ref{sec 2c}.

\section{Notations and Preliminaries}\label{sec 2}

Throughout we assume that $\Om$ is a bounded and smooth domain in $\R^n$. We fix $q> n+2$ and set, for $\kappa>1/q$,
$$
 \Wqb^\kappa:=\Wqb^\kappa (\Om):=\{u\in W_q^\kappa; u=0\ \text{on}\ \partial\Om\}\ ,
 $$
where $W_q^\kappa:=W_q^\kappa (\Om)$ stands for the usual Sobolev-Slobodeckii spaces (for arbitrary $\kappa>0$) and values on the boundary are interpreted in the sense of traces. Recall from \cite[III.Thm.4.10.2]{LQPP} and the Sobolev embedding theorem that
\bqn\label{emb}
\Wq\hookrightarrow C\big([0,a_m],\Wqb^{2-2/q}\big)\hookrightarrow C\big([0,a_m],C^1(\bar{\Om})\big)
\eqn 
so that the trace $\gamma_0 u:=u(0)\in \Wqb^{2-2/q}\hookrightarrow C^1(\bar{\Om})$ is well-defined for $u\in\Wq$. Moreover, since also
$$
\Wq\hookrightarrow W_q^1((0,a_m),L_q)\hookrightarrow C^{1-1/q}([0,a_m],L_q)
$$
with $L_q:=L_q(\Om)$, the interpolation inequality in \cite[I.Thm.2.11.1]{LQPP} yields in fact
\bqn\label{embb}
\Wq\hookrightarrow C^{1-1/q-\vartheta}([0,a_m],W_q^{2\vartheta})\ ,\quad 0\le \vartheta\le 1-1/q\ .
\eqn
Note that the interior, $\mathrm{int}(\Wqb^{2-2/q,+})$, of the positive cone $\Wqb^{2-2/q,+}$ of $\Wqb^{2-2/q}$ is not empty.
We set 
$$
\Lq:=L_q((0,a_m),L_q(\Om))\quad \text{and}\quad \Wqd:=\mathbb{W}_q^+\setminus\{0\}\ .
$$
Put $J:=[0,a_m]$. Given $\varrho>0$ and $h\in C^\varrho (J,C(\bar{\Om}))$, we let
$\Pi_{[h]}(a,\sigma)$, $0\le \sigma\le a\le a_m$,
denote the unique parabolic evolution operator corresponding to  
$-\Delta_D + h\in  C^\varrho \big(J,\ml(\Wqb^2,L_q)\big)$,
that is, $$z(a)=\Pi_{[h]}(a,\sigma)\Phi\ ,\quad a\in (\sigma,a_m)\ ,$$ defines the unique strong solution to 
$$
\partial_a z-\Delta_D z+hz=0\ ,\quad a\in (\sigma,a_m)\ ,\qquad z(\sigma)=\Phi\ ,
$$
for any given $\sigma\in (0,a_m)$ and $\Phi\in L_q$ (see \cite[II.Cor.4.4.1]{LQPP}). Note that the evolution operator is positive, i.e. 
$$
\Pi_{[h]}(a,\sigma)\Phi \in L_q^+\ ,\quad 0\le \sigma\le a\le a_m\ ,\quad \Phi\in L_q^+\ .
$$ 
As $J$ is a compact interval and $-\Delta_D$ has bounded imaginary powers, it follows from \eqref{emb} and \cite{LQPP} (in particular, see I.Cor.1.3.2, III.Thm.4.8.7, III.Thm.4.10.10 therein) that the operator $-\Delta_D + h$ has maximal $L_q$-regularity, i.e. the operator
$$
(\partial_a-\Delta_D+h,\gamma_0)\in\ml \big(\Wq,\Lq\times\Wqq\big)
$$
is a toplinear isomorphism. In particular,
$\Pi_{[h]}(\cdot,0)\Phi \in \Wq$ for $\Phi\in\Wqb^{2-2/q}$.
We set
$$
H_{[h]}:=\int_0^{a_m}b_1(a)\,\Pi_{[h]}(a,0)\,\rd a\ ,\quad \hat{H}_{[h]}:=\int_0^{a_m}b_2(a)\,\Pi_{[h]}(a,0)\,\rd a\ .
$$
Then $H_{[h]}$ and $\hat{H}_{[h]}$ belong to $\mk(\Wqb^{2-2/q})$, that is, they define compact linear operators on $\Wqb^{2-2/q}$, and they are strongly positive, that is, e.g.
\bqn\label{ssss}
H_{[h]}\Phi \in \mathrm{int}(\Wqqp)\ ,\quad \Phi\in\Wqqp\setminus\{0\}\ .
\eqn
The corresponding spectral radii $r(H_{[h]})$ and $r(\hat{H}_{[h]})$ can thus be characterized according to the Krein-Rutman theorem \cite[Thm.3.2]{AmannSIAMReview} (see Lemma~\ref{A2} from the appendix). In particular, the normalizations \eqref{6}  readily imply
\bqn\label{1mio}
r(H_{[0]})=r(\hat{H}_{[0]})=1\ 
\eqn
since any positive eigenfunction of $-\Delta_D$ is an eigenfunction of $H_{[h]}$ and $\hat{H}_{[h]}$ as well. 
It is worthwhile to point out that \eqref{embb} warrants an equivalent formulation of a solution $(u,v)\in\Wq\times\Wq$ to \eqref{1}-\eqref{4} as
\begin{align}
u(a)&=\Pi_{[\alpha_1 u\mp\alpha_2 v]}(a,0)\, u(0)\ ,\quad a\in J\ ,  &u(0)=\eta\, H_{[\alpha_1 u\mp\alpha_2 v]}\, u(0)\ ,\label{darst1}\\
 v(a)&=\Pi_{[\beta_1 v \mp\beta_2 u]}(a,0)\, v(0)\ ,\quad a\in J\ ,  &v(0)= \xi\, \hat{H}_{[\beta_1 v \mp\beta_2 u]}\, v(0)\ .\label{darst2}
\end{align}
Observe that $u$, $v$ are nonzero or nonnegative provided $u(0)$, $v(0)$ are nonzero or nonnegative. Hence, if $(u,v)\in\Wq^+\times\Wq^+$ solves \eqref{1}-\eqref{4}, then
\begin{align}
&\eta\, r(H_{[\alpha_1 u\mp\alpha_2 v]})=1\quad\text{if}\quad u(0)\in\Wqqp\setminus\{0\}\ ,\label{sp1}\\
&\xi\, r(\hat{H}_{[\beta_1 v \mp\beta_2 u]})=1\quad\text{if}\quad v(0)\in\Wqqp\setminus\{0\}\ ,\label{sp2}
\end{align}
owing to Lemma~\ref{A2}.
In particular, we have
\bqn\label{sp}
\eta\, r(H_{[\alpha_1 u_\eta]})=\xi\, r(\hat{H}_{[\beta_1 v_\xi]})=1\ ,\quad \eta,\xi>1\ 
\eqn
by \eqref{Aee1} and \eqref{Aee2} since $u_\eta(0), v_\xi(0)\in\Wqqp\setminus\{0\}$. We conclude this section with the following auxiliary result:

\begin{lem}\label{L00}
Given $M>0$ there is $c(M)>0$ such that
\bqn\label{iii}
\|u(a)\|_\infty+\|v(a)\|_\infty \le M\ ,\quad a\in J\ ,
\eqn
implies
\bqnn\label{bound}
\|u\|_{\Wq}\le c(M)(\eta+1)\ ,\quad \|v\|_{\Wq}\le c(M) (\xi+1)
\eqnn
for any solution $(u,v)\in\Wq^+\times\Wq^+$ to \eqref{1}-\eqref{4} with $\eta, \xi >0$.
\end{lem}

\begin{proof}
If $(u,v)\in\Wq^+\times\Wq^+$ solves \eqref{1}-\eqref{4}, we derive from \eqref{1}, \eqref{iii}, and the property of maximal $L_q$-regularity of $-\Delta_D$ that there is $c_0(M)>0$ such that
$$
\|u\|_{\Wq}\,\le\, c\,\big(\|u(0)\|_{\Wqq}+\|-\alpha_1 u^2\pm\alpha_2 uv\|_{\Lq}\big)\,\le\, c_0(M)\, \big(\|u(0)\|_{\Wqq}+1\big)\ .
$$
Writing \eqref{1} in the form
$$
u(a)=e^{a\Delta_D}u(0)+\int_0^a e^{(a-\sigma)\Delta_D}\big(-\alpha_1 u(\sigma)^2\pm\alpha_2 u(\sigma)v(\sigma)\big)\,\rd \sigma\ , \quad a\in J\ ,
$$
and using $\|e^{a\Delta_D}\|_{\ml(L_q,\Wqq)}\le c a^{1/q-1}$ for $a>0$, we obtain from \eqref{3} and \eqref{iii}
\bqnn
\begin{split}
\|u(0)\|_{\Wqq}\,&\le\,\eta\,\|b_1\|_\infty\int_0^{a_m} \|e^{a\Delta_D}\|_{\ml(L_q,\Wqq)}\, \|u(0)\|_{L_q}\,\rd a\\
&\quad+\eta\,\|b_1\|_\infty\int_0^{a_m}\int_0^a  \|e^{(a-\sigma)\Delta_D}\|_{\ml(L_q,\Wqq)}\, \big(\|\alpha_1 u(\sigma)^2\|_{L_q}+\|\alpha_2 u(\sigma) v(\sigma)\|_{L_q}\big)\,\rd \sigma\,\rd a\\
&\le\, c_1(M)\,\eta\ .
\end{split}
\eqnn
Consequently, $\|u\|_{\Wq}\le c(M)(\eta+1)$. Similarly we deduce $\|v\|_{\Wq}\le c(M)(\xi+1)$.

\end{proof}
 
\section{Cooperative Systems with $\xi<1$: Proof of Theorem~\ref{T2}}\label{sec 2b2}
We focus our attention on \eqref{1coo}-\eqref{2coo} subject to \eqref{3}-\eqref{4} when $\xi<1$.
First we show local bifurcation of a continuous curve from $\mathfrak{B}_1$ by using the results of Crandall-Rabinowitz~\cite{CrandallRabinowitz}.
We remark that 
\bqn\label{u}
\big(\eta\mapsto r(\hat{H}_{[-\beta_2 u_\eta]})\big)\in C\big( (1,\infty),(1,\infty)\big)\ \text{is strictly increasing}\ ,\qquad \lim_{\eta\rightarrow 1} r(\hat{H}_{[-\beta_2 u_\eta]})=1
\eqn
according to Lemma~\ref{A2}, Theorem~\ref{A1},  and \eqref{1mio}, so 
\bqn\label{nu}
\nu:=\dfrac{1}{ \lim\limits_{\eta\rightarrow \infty} r(\hat{H}_{[-\beta_2 u_\eta]})}\in [0,1) 
\eqn
is well-defined.

\begin{rem}\label{R32}
As $\|u_\eta\|_\infty\rightarrow \infty$ for $\eta\rightarrow\infty$ by Theorem~\ref{A1}, we conjecture $\nu=0$ in~\eqref{nu}.
\end{rem}

For the remainder of this section we fix $\xi\in (\nu,1)$. The observations above ensure the existence of a unique value $\eta_0:=\eta_0(\xi)>1$ for which
\bqn\label{xi1}
\xi\, r(\hat{H}_{[-\beta_2 u_{\eta_0}]})=1\ .
\eqn
Note then that
\bqn\label{z1}
\mathrm{ker}\big(1-\xi\hat{H}_{[-\beta_2 u_{\eta_0}]}\big)=\mathrm{span}\{\Psi_{0}\}\quad\text{with}\quad \Psi_{0}\in \mathrm{int}(\Wqqp)
\eqn
by the Krein-Rutman theorem. 
With these notations we have:

\begin{lem}\label{P22}
There is a local continuous curve $\mathfrak{C}_1\subset \R^+\times\Wqd\times\Wqd$ of coexistence solutions to \eqref{1coo}-\eqref{2coo} subject to \eqref{3}-\eqref{4} bifurcating from $(\eta_0,u_{\eta_0},0)\in \mathfrak{B}_1$, and all positive coexistence solutions near $(\eta_0,u_{\eta_0},0)$ lie on this curve.
\end{lem}

\begin{proof}
The proof is in the spirit of the one of \cite[Prop.2.7]{WalkerJRAM}. We are linearizing \eqref{1coo}-\eqref{2coo} around $(\eta_0,u_{\eta_0},0)\in \mathfrak{B}_1$. For this observe that $(\eta,u,v)=(\eta,u_\eta+w,v)\in \R\times\Wq\times \Wq$ solves \eqref{1coo}-\eqref{2coo} subject to \eqref{3}-\eqref{4} if and only if $(\eta,w,v)\in \R\times\Wq\times \Wq$ solves
\begin{align}
&\partial_a w-\Delta_D w=-\alpha_1w^2-2\alpha_1 u_\eta w+\alpha_2 v(u_\eta+w)\ ,& w(0)=\eta W\ ,\label{18aa}\\
&\partial_av-\Delta_Dv=-\beta_1 v^2+\beta_2 v(u_\eta+w)\ ,& v(0)=\xi V\ ,\label{18bb}
\end{align}
where we slightly abuse notation by writing
$$
W:=\int_0^{a_m} b_1(a)\, w(a)\,\rd a\ ,\qquad V:=\int_0^{a_m} b_2(a)\, v(a)\,\rd a\ 
$$
when $w,v\in\Wq$. We shall use this notation also for other capital letters since it will always be clear from the context, which of the profiles $b_1$ or $b_2$ is meant. Using maximal $L_q$-regularity of $-\Delta_D$, we may introduce the operator
$$
T:=\big(\partial_a-\Delta_D,\gamma_0\big)^{-1}\in\ml(\Lq\times\Wqq,\Wq)\ 
$$
so that the solutions to \eqref{18aa}-\eqref{18bb} are the zeros of the function
$$
F(\eta,w,u):=\left(\begin{matrix} w-T\big(-\alpha_1w^2-2\alpha_1 u_\eta w+\alpha_2 v(u_\eta+w)\, , \,\eta W\big)\\ v-T\big(-\beta_1 v^2+\beta_2 v(u_\eta+w)\, ,\, \xi V\big)\end{matrix}\right)\ .
$$
Observe that
$$
F\in C^2\big((1,\infty)\times\Wq\times\Wq,\Wq\times\Wq\big)
$$ 
with partial Frech\'et derivatives at $(\eta,w,v)=(\eta,0,0)$ given by 
$$
F_{(w,v)}(\eta,0,0)(\phi,\psi)=\left(\begin{matrix} \phi-T(-2\alpha_1 u_\eta \phi+\alpha_2 u_\eta\psi\, , \,\eta \Phi)\\ \psi-T( \beta_2 u_\eta\psi \, ,\, \xi \Psi)\end{matrix}\right)
$$
and
$$
F_{\eta,(w,v)}(\eta,0,0)(\phi,\psi)=\left(\begin{matrix} -T(-2\alpha_1 u'_\eta \phi+\alpha_2 u'_\eta\psi\, , \, \Phi)\\ -T( \beta_2 u'_\eta\psi \, ,\, 0)\end{matrix}\right)
$$
for $(\phi,\psi)\in \Wq\times\Wq$, where $u'_\eta:=\frac{\partial}{\partial\eta} u_\eta$ is well defined according to Theorem~\ref{A1}.
We next show that the kernel of $F_{(w,v)}(\eta_0,0,0)$ is one-dimensional.
Given $(\phi,\psi)\in \mathrm{ker}(F_{(w,v)}(\eta_0,0,0))$ we have
\begin{align}
\partial_a\phi-\Delta_D\phi &=-2\alpha_1 u_{\eta_0} \phi+\alpha_2 u_{\eta_0}\psi\ ,& \phi(0)=\eta_0\,\Phi\ ,\label{t1}\\
\partial_a\psi-\Delta_D\psi &=\beta_2 u_{\eta_0}\psi\ ,& \psi(0)=\xi\Psi\ .\label{tt1}
\end{align}
From \eqref{tt1} and \eqref{z1} we conclude that $\psi=\mu \psi_*$ for some $\mu\in\R$ with
$$
\psi_*:=\Pi_{[-\beta_2 u_{\eta_0}]}(\cdot,0)\,\Psi_{0}\in \Wq\ .
$$
Plugging this into \eqref{t1} and observing that $1-\eta_0 H_{[2\alpha_1 u_{\eta_0}]}$ is invertible since $$\eta_0 r(H_{[2\alpha_1 u_{\eta_0}]})<\eta_0 r(H_{[\alpha_1 u_{\eta_0}]})=1$$ by \eqref{sp}, Lemma~\ref{A2}, and the positivity of $u_{\eta_0}$, we derive $\phi=\mu \phi_*$, where
$$
\phi_{*}:=\Pi_{[2\alpha_1 u_{\eta_0}]}(\cdot,0)\Phi_{0}+ S\psi_{*}\in \Wq\ ,\quad a\in J\ , 
$$
with
$$(S\psi_{*})(a):=
{\alpha_2}\int_0^a \Pi_{[2\alpha_1 u_{\eta_0}]}(a,\sigma)\,\big( u_{\eta_0}(\sigma) \psi_{*}(\sigma)\big)\,\rd \sigma\ , \quad a\in J\  ,
$$
$$
\Phi_{0}:=\eta_0 \big(1-\eta_0 H_{[2\alpha_1 u_{\eta_0}]}\big)^{-1} \int_0^{a_m} b_1(a)(S\psi_{*})(a)\,\rd a\ .
$$
Therefore,
$$
\mathrm{ker}\big(F_{(w,v)}(\eta_0,0,0)\big)=\mathrm{span}\big\{(\phi_*,\psi_*)\big\}\ .
$$ 
As \cite[Thm.1.1]{AmannGlasnik00} and Sobolev's embedding theorem ensure the compact embedding $\Wq\dhr L_\infty(J,C(\bar{\Om}))$ since $q>n+2$, we have 
\bqn\label{compp}
\Wq\times\Wq\rightarrow \Lq\ ,\quad (w,v)\mapsto wv\quad\text{is compact}\ .
\eqn
In particular, it readily follows that the derivative of $F$ has the form $F_{(w,v)}(\eta_0,0,0)=1-\hat{T}$ with a compact operator $\hat{T}$. From this we get that also the codimension of $\mathrm{rg}\big(F_{(w,v)}(\eta_0,0,0)\big)$ equals one.
We next check the transversality condition of \cite{CrandallRabinowitz}. For, suppose that
$$
F_{\eta,(w,v)}(\eta_0,0,0)(\phi_*,\psi_*)\in\mathrm{rg}\big(F_{(w,v)}(\eta_0,0,0)\big)\ .
$$
Then there exists $v\in\Wq$ with 
$$
v-T(\beta_2 u_{\eta_0} v, \xi V)=-T(\beta_2 u'_{\eta_0} \psi_{*}\,,\,0)\ .
$$
Choosing $\tau>0$ such that $\tau\Psi_{0}-v(0)\in\mathrm{int}(\Wqqp)$ and putting $p:=\tau\psi_{*}-v$, we obtain from the definition of $\psi_*$
$$
p(a)=\Pi_{[-\beta_2 u_{\eta_0}]}(a,0) p(0)+\int_0^a \Pi_{[-\beta_2 u_{\eta_0}]}(a,\sigma)\, \big(\beta_2 u'_{\eta_0}(\sigma) \psi_{*}(\sigma)\big)\, \rd \sigma\ ,\quad a\in J\ .
$$
Thus, since $p(0)=\xi P$, 
$$
\big(1-\xi\hat{H}_{[-\beta_2 u_{\eta_0}]}\big) p(0)=\xi \int_0^{a_m} b_2(a)\int_0^a \Pi_{[-\beta_2 u_{\eta_0}]}(a,\sigma)\, \big(\beta_2 u'_{\eta_0}(\sigma) \psi_{*}(\sigma)\big)\, \rd \sigma\rd a\ .
$$
However, as the right hand side belongs to $\Wqqp\setminus\{0\}$ due to \eqref{5}, Theorem~\ref{A1}, and the strong positivity of the evolution operator $\Pi_{[-\beta_2 u_{\eta_0}]}(a,\sigma)$ on $\Wqq$ for $0\le\sigma <a\le a_m$,
this last equation admits no positive solution $p(0)$ according to \cite[Thm.3.2]{AmannSIAMReview} and \eqref{xi1} which clearly contradicts the fact that $p(0)=\tau\Psi_{0}-v(0)$ belongs to $\mathrm{int}(\Wqqp)$. Consequently,
$$
F_{\eta,(w,v)}(\eta_0,0,0)(\phi_{*},\psi_{*})\notin\mathrm{rg}\big(F_{(w,v)}(\eta_0,0,0)\big)\ .
$$
We are thus in a position to apply \cite[Thm.1.7]{CrandallRabinowitz} and deduce the existence of some $\ve_0>0$ and functions $\eta\in C((-\ve_0,\ve_0),\R)$ and $\theta_j\in C((-\ve_0,\ve_0),\Wq)$ with $\eta(0)=\eta_0$, $\theta_j(0)=0$ such that the nontrivial zeros of the function $F$ close to $(\eta_0,0,0)$ lie on the curve
$$
\big\{\big(\eta(\ve),\ve(\phi_*,\psi_*)+\ve(\theta_1(\ve),\theta_2(\ve))\big)\, ;\, \vert\ve\vert<\ve_0\big\}\ .
$$
By Theorem~\ref{A1},
$$
\mathfrak{C}_1:=\big\{\big(\eta(\ve),u_{\eta(\ve)}+\ve\phi_*+\ve\theta_1(\ve),\ve\psi_*+\ve\theta_2(\ve)\big)\, ;\, 0<\ve<\ve_0\big\}
$$
is then a continuous curve of solutions to \eqref{1coo}-\eqref{2coo}, \eqref{3}-\eqref{4} bifurcating from \mbox{ $(\eta_0,u_{\eta_0},0)\in\mathfrak{B}_1$}. As all traces $u_{\eta_0}(0)$, $\phi_*(0)=\Phi_{0}$, and $\psi_*(0)=\Psi_{0}$ belong to $\mathrm{int}(\Wqqp)$, it follows from \eqref{emb} and the continuity of $\theta_j$ that the initial values $u(0)$ and $v(0)$ for a solution $(\eta,u,v)\in \mathfrak{C}_1$ belong to $\mathrm{int}(\Wqqp)$ provided $\ve_0>0$ is sufficiently small, whence
$$
(u,v)\in\Wqd\times\Wqd\ ,\quad (\eta,u,v)\in \mathfrak{C}_1\ ,
$$ 
by \eqref{darst1}, \eqref{darst2}, and positivity of the corresponding evolution operators. This completes the proof of the lemma.
\end{proof}

Next we show that $\mathfrak{C}_1$ extends to a global continuum of positive coexistence solutions by invoking Rabinowitz' global alternative~\cite{Rabinowitz} along with the unilateral global results of L\'opez-G\'omez \cite{LopezGomezChapman}. The main steps of the proof are the same as in the proof of \cite[Thm.2.2]{WalkerJRAM}, but we have to argue here more subtle at several points  since 
we are deriving bifurcation with respect to the parameter $\eta$ by linearizing around a point $(\eta,u_\eta,0)\in\mathfrak{B}_1$ . \\

 

Setting $u_\eta:=0$ for $\eta\le 1$ it follows from Theorem~\ref{A1} that
\bqn\label{100}
\big(\eta\mapsto u_\eta\big)\in C(\R,\Wq^+)\ .
\eqn
Hence, defining
\begin{align*}
&Z_1[\eta]:=\big(\partial_a-\Delta_D+2\alpha_1u_\eta,\gamma_0)^{-1}\in \ml(\Lq\times\Wqq, \Wq)\ ,\\
&Z_2[\eta]:=\big(\partial_a-\Delta_D-\beta_2u_\eta,\gamma_0)^{-1}\in \ml(\Lq\times\Wqq, \Wq)\ ,
\end{align*}
based on maximal $L_q$-regularity (see Section~\ref{sec 2}), we deduce
\bqn\label{102}
\big(\eta\mapsto Z_j[\eta]\big)\in C\big(\R,\ml(\Lq\times \Wqq,\Wq)\big)\ ,\quad j=1, 2\ .
\eqn
Writing again $(\eta,u,v)=(\eta,u_\eta+w,v)\in \R\times\Wq\times \Wq$ and
recalling \eqref{18aa} and \eqref{18bb}, it follows that \eqref{1coo}-\eqref{2coo} subject to \eqref{3}-\eqref{4} may be recast equivalently as
\bqn\label{19t}
(w,v)-K(\eta)(w,v)+R(\eta,w,v)=0\ 
\eqn
by setting
$$
K(\eta)(w,v):=\left(\begin{matrix} Z_1[\eta](\alpha_2u_\eta v,\eta W)\\ Z_2[\eta](0,\xi V)\end{matrix}\right)\ ,\qquad R(\eta,w,v):=-\left(\begin{matrix} Z_1[\eta](-\alpha_1 w^2+\alpha_2 w v,0)\\ Z_2[\eta](-\beta_1v^2+\beta_2wv,0)\end{matrix}\right)\ 
$$
for $(w,v)\in \Wq\times\Wq$ still using the notation
\bqn\label{nota}
W:=\int_0^{a_m} b_1(a)\, w(a)\,\rd a\ ,\qquad V:=\int_0^{a_m} b_2(a)\, v(a)\,\rd a\ .
\eqn
This notation we shall use throughout the remainder of this section as no confusion seems likely.
Then \eqref{compp}, \eqref{100}, \eqref{102}, and the compact embedding $W_q^2\dhr \Wqq$ entail  
\bqn\label{K}
K(\eta)\in\mk(\Wq\times\Wq)\ \text{depends continuously on $\eta\in\R$}\ ,
\eqn
and 
\bqn\label{R}
R\in C(\R\times\Wq\times\Wq,\Wq\times\Wq)\ \text{is compact}
\eqn 
with
\bqn\label{RR}
R(\eta,w,v)=o\big(\|(w,v)\|_{\Wq\times\Wq}\big)\ \, \text{as}\ \, \|(w,v)\|_{\Wq\times\Wq}\rightarrow 0\ ,
\eqn
uniformly with respect to $\eta$ in compact intervals. Moreover, we have:

\begin{lem}\label{L6}
Let $\eta\in\R$. If $\mu\ge 1$ is an eigenvalue of the compact operator $K(\eta)$ with eigenvector $(w,v)\in\Wq\times\Wq$, then either $\eta>1$ and $\mu/\xi$ is an eigenvalue of~$\hat{H}_{[-\beta_2 u_\eta]}$ with eigenvector $v(0)\in\Wqq$ or $\mu=\eta=1$.
\end{lem}

\begin{proof}
Let $\mu\ge 1$ and $(w,v)\in (\Wq\times\Wq)\setminus\{(0,0)\}$ with $K(\eta)(w,v)=\mu (w,v)$. On the one hand, if $v= 0$, then
$$
\partial_a w-\Delta_Dw+2\alpha_1 u_\eta w=0\ ,\quad w(0)=\frac{\eta}{\mu} W\ ,
$$
from which
$$
w(a)=\Pi_{[2\alpha_1 u_\eta]}(a,0) w(0)\ ,\quad a\in J\ ,\qquad
w(0)= \frac{\eta}{\mu} H_{[2\alpha_1 u_\eta]} w(0)\ .
$$
In particular, $\eta\not= 0$ and $w(0)\not=0$ since otherwise $(w,v)=(0,0)$, and hence 
\bqn\label{104}
\mu\le \eta r(H_{[2\alpha_1 u_\eta]})\ .
\eqn
Next, $\eta>1$ is impossible since otherwise $u_\eta\in\Wqd$ and so
$$
\frac{\eta}{\mu}r(H_{[2\alpha_1 u_\eta]}) < \eta r(H_{[\alpha_1 u_\eta]}) =1
$$
by Lemma~\ref{A2} and \eqref{sp} contradicting \eqref{104}. Hence $\eta\le 1$ and thus
$$
\frac{\eta}{\mu}r(H_{[2\alpha_1 u_\eta]}) = \frac{\eta}{\mu}r(H_{[0]}) =\frac{\eta}{\mu}\le 1
$$
by \eqref{1mio} what is only possible if $\mu=\eta=1$ according to \eqref{104}. On the other hand, if $v\not= 0$, then from
$$
\partial_a v-\Delta_Dv-\beta_2u_\eta v=0\ ,\quad v(0)=\frac{\xi}{\mu}V\ 
$$
it follows
$$
v(a)=\Pi_{[-\beta_2 u_\eta]}(a,0) v(0)\ ,\quad a\in J\ ,\qquad
v(0)= \frac{\xi}{\mu} \hat{H}_{[-\beta_2 u_\eta]} v(0)\ ,
$$
and so $v(0)\not= 0$ and $\xi\not=0$ since otherwise $v= 0$. Consequently, $\mu/\xi$ is an eigenvalue of~$\hat{H}_{[-\beta_2 u_\eta]}$ with eigenvector $v(0)$. Assuming $\eta\le 1$ we have $u_\eta=0$ and thus $\mu/\xi\le r(\hat{H}_{[0]})=1$ by \eqref{1mio} contradicting $0< \xi<1$ and $\mu\ge 1$.
\end{proof}

As a consequence of Lemma~\ref{L6} the set of singular values of the family $K(\eta)$ is discrete:

\begin{cor}\label{C22}
The set $\Sigma:=\{\eta\in\R ; \mathrm{dim}(\mathrm{ker}(1-K(\eta)))\ge 1\}$ is discrete.
\end{cor}

\begin{proof}
Lemma~\ref{L6} ensures 
$$
\Sigma\cap (1,\infty)\subset \Xi:=\{\eta>1 ; \mathrm{dim}(\mathrm{ker}(1-\xi\hat{H}_{[-\beta_2 u_\eta]}))\ge 1\}\ .
$$
Due to
$$
\Pi_{[-\beta_2 u_\eta]}(\cdot,0)\Phi=\big(\partial_a-\Delta_D-\beta_2 u_\eta,\gamma_0)^{-1}(0,\Phi)\ ,\quad \Phi\in \Wqq\ ,\quad \eta>1\ ,
$$
it follows from the analyticity of the inversion map for linear operators and  the analyticity of the map $\eta\mapsto u_\eta$ stated in Theorem~\ref{A1} that also the map $(1,\infty)\rightarrow \mk(\Wqq)$, $ \eta\mapsto \xi\hat{H}_{[-\beta_2 u_\eta]}$ is real analytic. Thus, since $1-\xi\hat{H}_{[-\beta_2 u_\eta]}$ is invertible for $\eta\in (1,\eta_0)$ owing to \eqref{u} and \eqref{xi1}, we are in a position to apply \cite[Thm.4.4.4]{LopezGomezChapman} and conclude that $\Xi$ is discrete. If $\eta\in \Sigma$ with $\eta\le 1$, then necessarily $\eta=1$ by Lemma~\ref{L6}. 
\end{proof}

Next, we characterize the dependence on the parameter $\eta$ of the fixed point index $\mathrm{Ind}(0,K(\eta))$ of zero with respect to $K(\eta)$. Recall that $\mathrm{Ind}(0,K(\eta))=(-1)^{\zeta(\eta)}$, where $\zeta(\eta)$ is the sum of the algebraic multiplicities of all real eigenvalues of $K(\eta)$ greater than one, see e.g. \cite[Sect.5.6]{LopezGomezChapman}.

\begin{lem}\label{L8}
The fixed point index $\mathrm{Ind}(0,K(\eta))$ of zero with respect to $K(\eta)$ changes sign as $\eta$ crosses $\eta_0$.
\end{lem}

\begin{proof}
First, let $1<\eta<\eta_0$ and suppose there is an eigenvalue \mbox{$\mu>1$} of $K(\eta)$. Let  $(w,v)\in\Wq\times\Wq$ be a corresponding eigenvector. Then Lemma~\ref{L6} yields 
\bqn\label{105}
\mu\le \xi\, r(\hat{H}_{[-\beta_2 u_\eta]})\ .
\eqn
Since $\eta <\eta_0$ we have $u_\eta \le u_{\eta_0}$ owing to Theorem~\ref{A1}. But then, by Lemma~\ref{A2}, \eqref{105}, and the assumption $\mu>1$,
$$
1<\xi r(\hat{H}_{[-\beta_2 u_\eta]})< \xi r(\hat{H}_{[-\beta_2 u_{\eta_0}]})
$$
in contradiction to the definition of $\eta_0$ in \eqref{xi1}. Thus there is no eigenvalue $\mu>1$ of $K(\eta)$ if $1<\eta<\eta_0$, consequently
$$
\mathrm{Ind}(0,K(\eta))=1\ ,\quad 1<\eta<\eta_0\ .
$$
Next, it follows from \cite[II.Lem.5.1.4]{LQPP} and Theorem~\ref{A1} that the evolution operator $\Pi_{[-\beta_2 u_\eta]}(\cdot,0)$ depends continuously on $\eta$ (actually: analytically, cf. the proof of Corollary~\ref{C22}) and hence
$$
\hat{H}_{[-\beta_2 u_\eta]}\longrightarrow \hat{H}_{[-\beta_2 u_{\eta_0}]} \ \, \text{in}\ \, \mathcal{K}(\Wqq)\ \, \text{as}\ \, \eta\longrightarrow \eta_0\ .
$$
According to \cite[IV.$\S$3.5]{Kato},
$$
\lambda_2\big(\xi\hat{H}_{[-\beta_2 u_\eta]}\big)\longrightarrow \lambda_2\big(\xi\hat{H}_{[-\beta_2 u_{\eta_0}]}\big) < r(\xi\hat{H}_{[-\beta_2 u_{\eta_0}]})=1 \ \, \text{as}\ \, \eta\longrightarrow \eta_0\ ,
$$
with $\lambda_2(H)$ denoting the second eigenvalue of a compact operator $H$. Choose $\ve>0$ with
\bqn\label{107}
\lambda_2\big(\xi\hat{H}_{[-\beta_2 u_\eta]}\big) < 1-\ve\ ,\quad \eta_0<\eta<\eta_0+\ve\ .
\eqn
Let $\eta_0\le\eta<\eta_0+\ve$ and let $\mu\ge 1$ be an eigenvalue of $K(\eta)$. Then $\mu\ge 1$ is an eigenvalue of $\xi \hat{H}_{[-\beta_2 u_\eta]}$ due to Lemma~\ref{L6} and thus $\mu=r(\xi\hat{H}_{[-\beta_2 u_\eta]})=:\mu_*$ since $\mu_*$ is the only eigenvalue in $(1-\ve,\infty)$ by \eqref{107}. But $\mu_*$ is a simple eigenvalue of $K(\eta)$. Indeed, noticing that
$$
K(\eta) (\phi,\psi)=\mu_* (\phi,\psi)
$$
is equivalent to
\begin{align*}
\partial_a\phi-\Delta_D\phi &=-2\alpha_1 u_{\eta} \phi+\frac{\alpha_2}{\mu_*} u_{\eta}\psi\ ,& \phi(0)=\frac{\eta}{\mu_*}\,\Phi\ ,\\
\partial_a\psi-\Delta_D\psi &=\beta_2 u_{\eta}\psi\ ,& \psi(0)=\frac{\xi}{\mu_*}\Psi\ ,
\end{align*}
it follows as in the proof of Lemma~\ref{L6} (see \eqref{t1} and \eqref{tt1}) that
$$
\mathrm{ker}\big(K(\eta)-\mu_*\big)=\mathrm{span}\{(\phi_{*},\psi_{*})\}\ ,
$$
where
$$
\psi_{*}:=Z_2[\eta](0,\Psi_1)=\Pi_{[-\beta_2 u_\eta]}(\cdot,0)\, \Psi_1\in \Wqd
$$
with $\Psi_1\in\mathrm{int}(\Wqqp)$ spanning $\mathrm{ker}(\mu_*-\xi\hat{H}_{[-\beta_2 u_\eta]})$ and
\begin{align*}
\Phi_1&:=\frac{\eta}{\mu_*}
\big(1-\frac{\eta}{\mu_*}H_{[2\alpha_1 u_\eta]}\big)^{-1} \int_0^{a_m} b_1(a)(S\psi_{*})(a)\,\rd a\in\Wqqp\  ,\\ 
\phi_{*}&:=\Pi_{[2\alpha_1 u_\eta]}(\cdot,0)\Phi_1+ S\psi_{*} =Z_1[\eta](S\psi_{*},\Phi_1)\in \Wq^+ . 
\end{align*}
Invertibility of $1-\frac{\eta}{\mu_*}H_{[2\alpha_1 u_\eta]}$ is due to $\mu_*\ge 1$, \eqref{sp}, and Lemma~\ref{A2}.
It then merely remains to prove that $\mu_*$ is simple. For, let $(\phi_{*},\psi_{*})\in\mathrm{rg}(K(\eta)-\mu_*)$. Then $$Z_2[\eta](0,\xi V)-\mu_*v=\psi_{*}$$ for some $v\in\Wq$, that is,
$$
\partial_av-\Delta_D v-\beta_2 u_\eta v=-\frac{1}{\mu_*}\big(\partial_a\psi_{*}-\Delta_D\psi_{*}-\beta_2 u_\eta \psi_{*}\big)=0\ ,\quad v(0)=\frac{\xi}{\mu_*}V-\frac{1}{\mu_*}\Psi_1\ .
$$
This readily implies
$$
\big(1-\frac{\xi}{\mu_*}\hat{H}_{[-\beta_2 u_\eta]}\big)\, v(0)=-\frac{1}{\mu_*}\Psi_1
$$
so that 
$$
\Psi_1\in \mathrm{ker}\big(1-\frac{\xi}{\mu_*}\hat{H}_{[-\beta_2 u_\eta]}\big)\cap \mathrm{rg}\big(1-\frac{\xi}{\mu_*}\hat{H}_{[-\beta_2 u_\eta]}\big)\
$$
contradicting the fact that the intersection equals $\{0\}$ since $\mu_*/\xi=r(\hat{H}_{[-\beta_2 u_\eta]})$ is a simple eigenvalue of $\hat{H}_{[-\beta_2 u_\eta]}$. Thus $(\phi_{*},\psi_{*})\notin\mathrm{rg}(K(\xi)-\mu_*)$ and $\mu_*$ is indeed a simple eigenvalue of $K(\eta)$. This ensures
$$
\mathrm{Ind}(0,K(\eta))=-1\ ,\quad \eta_0\le\eta<\eta_0+\ve\ ,
$$
and the assertion follows.
\end{proof}

Recalling the definition of $\Psi_0$ in \eqref{z1} and taking $\eta=\eta_0$ (and so $\mu_*=1$), the proof of Lemma~\ref{L8} reveals:

\begin{cor}\label{C3}
$\mu_\star=1$ is a simple eigenvalue of $K(\eta_0)$. Thus
$$
\Wq\times\Wq=\mathrm{ker}\big(1-K(\eta_0)\big)\oplus \mathrm{rg}\big(1-K(\eta_0)\big)\ ,\qquad \mathrm{ker}\big(1-K(\eta_0)\big)=\mathrm{span}\{(\phi_{\star},\psi_{\star})\} 
$$
with $\psi_{\star}=Z_2[\eta_0](0,\Psi_0)\in\Wqd$, $\Psi_0=\xi\Psi_\star\in\mathrm{int}(\Wqqp)$, and $\phi_{\star}\in\Wqd$.
\end{cor}

Corollary~\ref{C22} and Lemma~\ref{L8} warrant that we may apply Rabinowitz' global alternative \cite[Cor.6.3.2]{LopezGomezChapman} to \eqref{19t}. Hence, we obtain a continuum $\mathfrak{C}'_1$ of solutions $(\eta,u,v)$ in $\R\times\Wq\times\Wq$ to \eqref{1coo}-\eqref{2coo} subject to \eqref{3}-\eqref{4} emanating from $(\eta_0,u_{\eta_0},0)\in\mathfrak{B}_1$. In combination with the unilateral global bifurcation result \cite[Thm.6.4.3]{LopezGomezChapman} and Corollary~\ref{C3} we derive that $\mathfrak{C}'_1$ satisfies the alternatives
\begin{itemize}
\item[(i)] $\mathfrak{C}'_1$ is unbounded in $\R\times\Wq\times\Wq$, or
\item[(ii)] there is $\eta\in\Sigma\setminus\{\eta_0\}$ with $(\eta,u_\eta,0)\in \mathfrak{C}'_1$, or
\item[(iii)] there is $(\eta,u_{\eta}+w,v)\in\mathfrak{C}'_1$ with $(w,v)\in\mathrm{rg}(1-K(\eta_0))\setminus\{(0,0)\}$.
\end{itemize}
By Lemma~\ref{P22}, $\mathfrak{C}'_1$ close to $(\eta_0,u_{\eta_0},0)$ coincides with $\mathfrak{C}_1\subset \R^+\times\Wqd\times\Wqd$ suggesting to abuse notation by putting
$$
\mathfrak{C}_1:=\mathfrak{C}'_1\cap (\R^+\times\Wqd\times\Wqd)\not=\emptyset\ .
$$
In fact, we have:

\begin{lem}\label{L33}
$\mathfrak{C}_1$ is unbounded in $\R^+\times\Wqd\times\Wqd$.
\end{lem}

\begin{proof}
Suppose $\mathfrak{C}'_1$ leaves $\R^+\times\Wqd\times\Wqd$ at some point $(\eta,u,v)\in \mathfrak{C}'_1$ different from $(\eta_0,u_{\eta_0},0)$ and let
$(\eta_j,u_j,v_j)\in \mathfrak{C}_1$
such that
$$
(\eta_j,u_j,v_j)\longrightarrow (\eta,u,v)\quad\text{in}\quad \R\times \Wq\times\Wq\ .
$$
Since obviously $\eta\ge 0$, $u\ge 0$, and $v\ge 0$, the only possibility is that $u= 0$ or $v= 0$. However, as the only solutions in $\R^+\times\Wq^+\times\Wq^+$ close to $\mathfrak{B}_0$ lie on $\mathfrak{B}_1$, the case $(u,v)= (0,0)$ is impossible since $v_j\not= 0$. If $u= 0$ but $v\not= 0$, then $(\eta,u,v)=(\eta,0,v)\in\mathfrak{C}'_1$ and $v\in\Wqd$ solves
$$
\partial_a v-\Delta_D v=-\beta_1 v^2\ ,\quad v(0)=\xi V\ 
$$
with $\xi<1$ contradicting Theorem~\ref{A1}. Consequently, $v= 0$ but $u\not= 0$, thus $u\in\Wqd$ solves
$$
\partial_a u-\Delta_D u=-\alpha_1 u^2\ ,\quad u(0)=\eta U\ ,
$$
whence $\eta>1$ and $u=u_\eta$ by Theorem~\ref{A1}. Therefore, $(\eta,u,v)=(\eta,u_\eta,0)\in\mathfrak{C}'_1$ and we may assume that $\eta_j>1$.
To demonstrate that this also leads to a contradiction, we adapt an argument of \cite[Thm.3.1]{BlatBrown2}. Put $z_j:=(w_j,v_j)$, where $w_j:=u_j-u_{\eta_j}$, and note that $z_j\rightarrow (0,0)$ as $j\rightarrow \infty$ by the previous observation. Moreover, since $u_j,v_j\ge 0$ and $\eta_j>1$, we obtain from
$$
\partial_a u_j-\Delta_D u_j=-\alpha_1 u_j^2+\alpha_2 v_j u_j\ge -\alpha_1 u_j^2\ ,\quad u_j(0)=\eta_j U_j\ ,
$$
that $u_j\ge u_{\eta_j}$ by invoking Lemma~\ref{A2}, whence $z_j\in\Wq^+\times\Wq^+$.
We then define 
$$
Q:\R\times\Wq^2\rightarrow\Wq^2\ ,\quad Q(\zeta,z):=K(\zeta)z-R(\zeta,z)
$$ 
and observe that $Q$ is differentiable with respect to $z\in\Wq^2$, $Q(\zeta,0)=0$ for $\zeta\in\R$, and $Q(\eta_j,z_j)=z_j$. The mean value theorem ensures
$$
z_j-Q_z(\eta,0)z_j=\int_0^1\big[Q_z(\eta_j,sz_j)z_j-Q_z(\eta,0)z_j\big]\,\rd s
$$
and hence, setting $m_j:=z_j/\|z_j\|_{\Wq^2}\in\Wq^+\times\Wq^+$ and taking $Q_z(\eta,0)=K(\eta)$ into account,
$$
m_j-K(\eta)m_j=\int_0^1\big[Q_z(\eta_j,sz_j)m_j-Q_z(\eta,0)m_j\big]\,\rd s\longrightarrow 0\ \, \text{as}\ \, j\rightarrow\infty
$$
by the boundedness of $(m_j)_{j\in\N}$, $(\eta_j,z_j)\rightarrow (\eta,0)$, and Lebesgue's theorem. As $K(\eta)$ is compact, this readily implies the existence of $m\in\Wq^+\times\Wq^+$ with $\|m\|_{\Wq^2}=1$ and $m=K(\eta)m$. Owing to Lemma~\ref{L6} we conclude that $1/\xi$ is an eigenvalue of $\hat{H}_{[-\beta_2 u_\eta]}$ with positive eigenvector. Hence $1=\xi r(\hat{H}_{[-\beta_2 u_\eta]})$ due to the Krein-Rutman theorem (see Lemma~\ref{A3a}) yielding \mbox{$\eta=\eta_0$} what is impossible since $(\eta,u,v)$ then coincides with $(\eta_0,u_{\eta_0},0)$. Therefore, $\mathfrak{C}_1=\mathfrak{C}_1'$ does not leave $\R^+\times\Wqd\times\Wqd$ except at $(\eta_0,u_{\eta_0},0)$. 

As a consequence of the preceding observation, alternative (ii) above can be ruled out. Suppose then that alternative (iii) above occurs, i.e. let $(\eta,u_{\eta}+w,v)\in\mathfrak{C}'_1$ be such that
$$
(0,0)\not=(w,v)=(1-K(\eta_0)) (f,g)\ 
$$
for some $(f,g)\in\Wq\times\Wq$. To derive a contradiction we argue similarly as in the proof of Lemma~\ref{P22}.
As $v\in\Wqd$, we have $v(0)=\xi V\in\Wqqp\setminus\{0\}$. Recall $\psi_{\star}(0)=\Psi_0\in\mathrm{int}(\Wqqp)$ from Corollary~\ref{C3} so that we may choose $\tau>0$ with 
$$
g(0)-v(0)+\tau\Psi_0\in\mathrm{int}(\Wqqp)\ .
$$ 
Note that
$$
v=g-Z_2[\eta_0](0,\xi G)\ ,\quad \psi_{\star}=Z_2[\eta_0](0,\xi\Psi_{\star})\ ,\quad p:=g-v+\tau\psi_{\star}=Z_2[\eta_0]\big(0,\xi(G+\tau\Psi_{\star})\big)\ .
$$
The last equality reads
$$
\partial_a p-\Delta_D p-\beta_2u_{\eta_0} p=0\ ,\quad p(0)=\xi(G+\tau\Psi_{\star})=\xi P +\xi V\ ,
$$
from which we deduce that
$$
\big(1-\xi \hat{H}_ {[-\beta_2 u_{\eta_0}]}\big) p(0)=\xi V \in\Wqqp\setminus\{0\}
$$
with $p(0)\in\mathrm{int}(\Wqqp)$ by the choice of $\tau$. However, this equation has no positive solution owing to \cite[Thm.3.2]{AmannSIAMReview} and the definition of $\xi$ in \eqref{xi1}. This shows that alternative (iii) above is impossible as well and the only remaining possibility is that $\mathfrak{C}_1=\mathfrak{C}'_1$ is unbounded in $\R^+\times\Wqd\times\Wqd$. 
\end{proof}

We remark that the bifurcation point $(\eta_0,u_{\eta_0},0)$ is unique:

\begin{cor}\label{C13}
There is no other bifurcation point on $\mathfrak{B}_1$ to positive coexistence solutions than $(\eta_0,u_{\eta_0},0)$.
\end{cor}

\begin{proof}
Suppose $(\eta,u_{\eta},0)\in \mathfrak{B}_1$ is a bifurcation point to positive coexistence solutions. Approximating this point by positive solutions we derive as in the proof of Lemma~\ref{L33} that $1$ is an eigenvalue of $K(\eta)$ with an eigenvector in \mbox{$\Wq^+\times\Wq^+$} so that, according to Lemma~\ref{L6}, $1/\xi$ is an eigenvalue of $\hat{H}_{[-\beta_2 u_\eta]}$ with positive eigenvector. As above, this implies $1=\xi r(\hat{H}_{[-\beta_2 u_\eta]})$ due to the Krein-Rutman theorem (see Lemma~\ref{A3a}), whence \mbox{$\eta=\eta_0$}.
\end{proof}

This completes the proof of Theorem~\ref{T2}. It remains to give a more precise characterization of the global nature of $\mathfrak{C}_1$ as stated in Corollary~\ref{C0}:

\begin{cor}\label{L3}
The continuum $\mathfrak{C}_1$ is unbounded with respect to both the parameter $\eta$ and the $u$-component in $\Wq$, or with respect to the $v$-component in $\Wq$. If \eqref{B} holds
for some $s>0$, then $\mathfrak{C}_1$ is unbounded with respect to the $u$-component in $\Wq$.
\end{cor}

\begin{proof}
(i) We have $u\ge u_\eta$ for any $(\eta,u,v)\in\mathfrak{C}_1$ with $\eta>1$ by the comparison principle of Lemma~\ref{A3} since
$$
\partial_a u-\Delta_D u=-\alpha_1 u^2+\alpha_2 vu\ge-\alpha_1 u^2\ ,\quad u(0)=\eta U\ .
$$ 
Since $\| u_\eta (0)\|_\infty\rightarrow\infty$ as $\eta\rightarrow\infty$ according to Theorem~\ref{A1}, we conclude that $\mathfrak{C}_1$ is unbounded with respect to $\eta$ only if it is unbounded with respect to the $u$-component in $\Wq$. 

(ii) Next suppose \eqref{B} and that there is $M>s/\beta_2$ such that $\|u(a)\|_\infty\le M$, $a\in J$, for all $(\eta,u,v)\in \mathfrak{C}_1$. Noticing
$$
\partial_a v-\Delta_D v=-\beta_1 v^2+\beta_2 u v\le -\beta_1 v^2+\beta_2 M v\ ,\,\quad  a\in (0,a_m)\ ,\quad x\in\Om\ ,
$$
it follows from the parabolic maximum principle \cite[Thm.13.5]{DanersKochMedina} that $v(a)\le f(a)$ on $\bar{\Om}$ for $a\in J$, where 
$$
f(a)\,:=m\,\|v(0)\|_\infty\,\left(\beta_1\,\|v(0)\|_\infty\big(1-e^{-ma}\big)+me^{-ma}\right)^{-1}\ ,\quad a\in J\ ,
$$
with $m:=\beta_2 M>s$ satisfies
$$
f'(a)=-\beta_1 f^2(a)+mf(a)\ ,\quad a\in J\ ,\qquad f(0)=\|v(0)\|_\infty\ .
$$ 
Thus \eqref{4} and \eqref{B} imply
$$
v(0)=\xi V\le  \xi \int_0^{a_m} b_2(a) f(a)\, \rd a \le \frac{\xi m}{\beta_1}\int_0^{a_m} b_2(a) (1-e^{-sa})^{-1}\, \rd a <\infty\quad\text{on}\ \bar{\Om}\ ,
$$
and so, owing to the definition of $f$, there is some $c>0$ such that $\|v(a)\|_\infty\le c$, $a\in J$ for all $(\eta,u,v)\in \mathfrak{C}_1$. Hence, $\mathfrak{C}_1$ is bounded with respect to the $v$-component by Lemma~\ref{L00} contradicting our findings in~(i). Consequently, if \eqref{B} holds, then $\mathfrak{C}_1$ is unbounded with respect to the $u$-component in $\Wq$.
\end{proof}

\section{Cooperative Systems with $\xi>1$: Proof of Theorem~\ref{T1}}\label{sec 2b1}

We still focus our attention on \eqref{1coo}-\eqref{2coo} subject to \eqref{3}-\eqref{4}, but let now $\xi>1$ be arbitrarily fixed for the remainder of this section and put
 \bqn\label{17}
 \eta_1:=\eta_1(\xi):=\frac{1}{r(H_{[-\alpha_2 v_\xi]})}\ .
\eqn
Then $\eta_1\in (0,1)$ according to \eqref{1mio} and Lemma~\ref{A2}. The Krein-Rutman theorem ensures
\bqn\label{z}
\mathrm{ker}\big(1-\eta_1H_{[-\alpha_2 v_\xi]}\big)=\mathrm{span}\{\Phi^0\}\quad\text{with}\quad \Phi^0\in \mathrm{int}(\Wqqp)\ .
\eqn
We first prove local bifurcation of a continuous curve from $(\eta_1(\xi),0,v_\xi)\in\mathfrak{B}_2$ by invoking the theorem of Crandall-Rabinowitz \cite{CrandallRabinowitz}. The present situation, however, turns out to be simpler than in the previous section.

\begin{lem}\label{L1}
A local continuous curve $\mathfrak{C}_2$ of positive coexistence solutions to \eqref{1coo}-\eqref{2coo} subject to \eqref{3}-\eqref{4} bifurcates from $(\eta_1(\xi),0,v_\xi)\in\mathfrak{B}_2$, and all positive coexistence solutions near  $(\eta_1(\xi),0,v_\xi)$ lie on this curve.
\end{lem}

\begin{proof}
We proceed similar to the proof of Lemma~\ref{P22}. Writing solutions to \eqref{1coo}-\eqref{2coo} subject to \eqref{3}-\eqref{4} in the form $(\eta,u,v)=(\eta,u,v_\xi+w)\in \R\times\Wq\times \Wq$, we have
\begin{align}
&\partial_a u-\Delta_D u=-\alpha_1u^2+\alpha_2 u(v_\xi+w)\ ,& u(0)=\eta U\ ,\label{18a}\\
&\partial_aw-\Delta_Dw=-\beta_1w^2-2\beta_1 v_\xi w+\beta_2 u(v_\xi+w)\ ,& w(0)=\xi W\ ,\label{18b}
\end{align}
where we agree upon the notation (and similarly for other capital letters)
$$
U:=\int_0^{a_m}b_1(a) u(a)\,\rd a\ ,\quad W:=\int_0^{a_m}b_2(a) w(a)\,\rd a\ .
$$
Thus we are lead to examine the zeros of the function $G\in C^2((1,\infty)\times\Wq\times\Wq,\Wq\times\Wq)$ given by
$$
G(\eta,u,w):=\left(\begin{matrix} u-T(-\alpha_1u^2+\alpha_2 u(v_\xi+w)\, , \,\eta U)\\ w-T(-\beta_1 w^2 -2\beta_1 v_\xi w+\beta_2 u (v_\xi+w)\, ,\, \xi W)\end{matrix}\right)\ ,
$$
with $T$ as in the proof of Lemma~\ref{P22}. For the partial Frech\'et derivatives at $(\eta,u,w)=(\eta,0,0)$ we compute
$$
G_{(u,w)}(\eta,0,0)(\phi,\psi)=\left(\begin{matrix} \phi-T(\alpha_2 \phi v_\xi\, , \,\eta \Phi)\\ \psi-T( -2\beta_1 v_\xi \psi+\beta_2 \phi v_\xi\, ,\, \xi \Psi)\end{matrix}\right)
$$
and
$$
G_{\eta,(u,w)}(\eta,0,0)(\phi,\psi)=\left(\begin{matrix} -T(0, \Phi)\\ 0\end{matrix}\right) 
$$
for $(\phi,\psi)\in \Wq\times\Wq$. Arguments similar to the ones in the proof of Lemma~\ref{P22} yield
$$
\mathrm{ker}\big(G_{(u,w)}(\eta_1,0,0)\big)=\mathrm{span}\big\{(\phi_{\varstar},\psi_{\varstar})\big\}\ ,
$$ 
where (see \eqref{z} and Lemma~\ref{A2}) 
\bqn\label{phistern}
\phi_{\varstar}:=\Pi_{[-\alpha_2 v_\xi]}(\cdot,0)\,\Phi^0\in \Wq^+\ ,
\eqn
and
$$
\psi_{\varstar}:=\Pi_{[2\beta_1 v_\xi]}(\cdot,0)\Psi^0+ S\phi_{\varstar}\in \Wq\ ,\qquad \Psi^0:=\xi \big(1-\xi \hat{H}_{[2\beta_1 v_\xi]}\big)^{-1} \int_0^{a_m} b_2(a)(S\phi_{\varstar})(a)\,\rd a\  ,
$$
with
$$
(S\phi_{\varstar})(a):=
{\beta_2}\int_0^a \Pi_{[2\beta_1 v_\xi]}(a,\sigma)\,\big(v_\xi(\sigma) \phi_{\varstar}(\sigma)\big)\,\rd \sigma\ , \quad a\in J\ .
$$ 
Observing that the derivative of $G$ has the form $G_{(u,w)}(\eta_1,0,0)=1-\hat{T}$ with a compact operator $\hat{T}$ (see \eqref{compp}), we get that also the codimension of $\mathrm{rg}\big(G_{(u,w)}(\eta_1,0,0)\big)$ equals one.
Next assume that
$$
G_{\eta,(u,w)}(\eta_1,0,0)(\phi_{\varstar},\psi_{\varstar})\in\mathrm{rg}\big(G_{(u,w)}(\eta_1,0,0)\big)
$$
and let $u\in\Wq$ be with 
$$
u-T(\alpha_2 v_\xi u, \eta_1 U)=-T(0, \Phi_{\varstar})\ .
$$
Then
$$
\partial_a u-\Delta_D u =\alpha_2 v_\xi u\ ,\quad a\in J\ ,\qquad u(0)=\eta_1U-\Phi_{\varstar}\ .
$$ 
This readily implies 
$$
\Phi_{\varstar}=-\big(1-\eta_1 H_{[-\alpha_2 v_\xi]}\big) u(0)\in\mathrm{rg}\big(1-\eta_1 H_{[-\alpha_2 v_\xi]}\big)
$$
in contradiction to
$$
\Phi_{\varstar}=\frac{1}{\eta_1}\Phi^0\in\mathrm{ker}\big(1-\eta_1 H_{[-\alpha_2 v_\xi]}\big)
$$
by \eqref{z} and \eqref{phistern} since $\eta_1 r(H_{[-\alpha_2 v_\xi]})=1$ is a simple eigenvalue of the compact operator $\eta_1 H_{[-\alpha_2 v_\xi]}$. Consequently, 
$$
G_{\eta,(u,w)}(\eta_1,0,0)(\phi_{\varstar},\psi_{\varstar})\notin\mathrm{rg}\big(G_{(u,w)}(\eta_1,0,0)\big),
$$
and we may again apply \cite[Thm.1.7]{CrandallRabinowitz}. Thus, the nontrivial zeros of the function $G$ lie on the curve
$$
\big\{\big(\eta(\ve),\ve(\phi_{\varstar},\psi_{\varstar})+\ve(\theta_1(\ve),\theta_2(\ve))\big)\, ;\, \vert\ve\vert<\ve_0\big\}\ ,
$$
for some $\ve_0>0$ and functions $\eta\in C((-\ve_0,\ve_0),\R)$ and $\theta_j\in C((-\ve_0,\ve_0),\Wq)$ with $\eta(0)=\eta_1$, $\theta_j(0)=0$. 
Thus,
$$
\mathfrak{C}_2:=\big\{\big(\eta(\ve),\ve\phi_{\varstar}+\ve\theta_1(\ve),v_\xi+\ve\psi_{\varstar}+\ve\theta_2(\ve)\big)\, ;\, 0<\ve<\ve_0\big\}
$$
defines a continuous curve of solutions to \eqref{1coo}-\eqref{2coo}, \eqref{3}-\eqref{4} bifurcating from \mbox{ $(\eta_1,0,v_\xi)\in\mathfrak{B}_2$}. As $\phi_{\varstar}(0)=\Phi^0\in\mathrm{int}(\Wqqp)$ and $v_\xi(0)\in \mathrm{int}(\Wqqp)$, it follows from \eqref{darst1} and \eqref{darst2} that $$(u,v)\in\Wqd\times\Wqd\ ,\quad (\eta,u,v)\in \mathfrak{C}_2\ ,
$$ 
provided $\ve_0>0$ is sufficiently small. This completes the proof of the lemma.
\end{proof}

To prove the assertion on global bifurcation of Theorem~\ref{T1} we invoke Rabinowitz' global alternative~\cite{Rabinowitz} and the unilateral global theorem ~\cite{LopezGomezChapman} as in the proof of Theorem~\ref{T2}. Again, the present situation is considerably simpler than in the proof of Theorem~\ref{T2}. 

\begin{lem}\label{L2}
The local curve $\mathfrak{C}_2$ extends to an unbounded continuum of coexistence solutions $(\eta,u,v)$ in $\R^+\times\Wqd\times\Wqd$ to \eqref{1coo}-\eqref{2coo} subject to \eqref{3}-\eqref{4}.
\end{lem}

\begin{proof}
Introducing the operators
\begin{align*}
&\tilde{Z}_1:=\big(\partial_a-\Delta_D-\alpha_2 v_\xi,\gamma_0)^{-1}\in \ml(\Lq\times\Wqq, \Wq)\ ,\\
&\tilde{Z}_2:=\big(\partial_a-\Delta_D+2\beta_1 v_\xi,\gamma_0)^{-1}\in \ml(\Lq\times\Wqq, \Wq)\ ,
\end{align*}
we may rewrite \eqref{18a}-\eqref{18b} equivalently as
\bqn\label{19}
(u,w)-\tilde{K}(\eta)(u,w)+\tilde{R}(u,w)=0\ 
\eqn
by setting
$$
\tilde{K}(\eta)(u,w):=\left(\begin{matrix} \tilde{Z}_1(0,\eta U)\\ \tilde{Z}_2(\beta_2 u v_\xi,\xi W)\end{matrix}\right)\ ,\qquad \tilde{R}(u,w):=-\left(\begin{matrix} \tilde{Z}_1(\alpha_1 u^2+\alpha_2uw,0)\\ \tilde{Z}_2(-\beta_1w^2+\beta_2uw,0)\end{matrix}\right)\ 
$$
for $(u,w)\in \Wq\times\Wq$. It is now easy to check on the basis of the previous section that the analogues of  \eqref{K}, \eqref{R}, \eqref{RR}, and accordingly Lemma~\ref{L6}, Corollary~\ref{C22}, Lemma~\ref{L8}, and Corollary~\ref{C3} hold for $\tilde{K}$ and $\tilde{R}$ when replacing $\eta_0$ by $\eta_1$. Consequently, we may apply the results \cite[Cor.6.3.2,Lem.6.4.1,Thm.6.4.3]{LopezGomezChapman} on unilateral global bifurcation to \eqref{19} and thus derive the existence of a continuum $\mathfrak{C}'_2$ of solutions $(\eta,u,v)$ to \eqref{1coo}-\eqref{2coo} subject to \eqref{3}-\eqref{4} in $\R\times\Wq\times\Wq$ emanating from $(\eta_1,0,v_\xi)$ and
satisfying the alternatives
\begin{itemize}
\item[(i)] $\mathfrak{C}'_2$ is unbounded in $\R\times\Wq\times\Wq$, or
\item[(ii)] $\mathfrak{C}'_2$ contains a point $(\eta,0,v_\xi)$ with $\eta\in\{\zeta\in\R ; \mathrm{dim}(\mathrm{ker}(1-\tilde{K}(\zeta)))\ge 1\}$ and $\eta\not=\eta_1$, or
\item[(iii)] $\mathfrak{C}'_2$ contains a point $(\eta,u,v_\xi+w)$ with $(u,w)\in\mathrm{rg}(1-\tilde{K}(\eta_1))$ with $(u,w)\not= (0,0)$.
\end{itemize}
By Lemma~\ref{L1}, $\mathfrak{C}'_2$ coincides with $\mathfrak{C}_2$ near $(\eta_1,0,v_\xi)$ suggesting to abuse notation by putting
$$
\mathfrak{C}_2:=\mathfrak{C}'_2\cap (\R^+\times\Wqd\times\Wqd)\not=\emptyset\ .
$$
We then claim that this so defined continuum
$\mathfrak{C}_2$ is unbounded in $\R^+\times\Wqd\times\Wqd$.
Indeed, suppose $\mathfrak{C}_2'$ leaves $\R^+\times\Wqd\times\Wqd$ at some point $(\eta,u,v)\in \mathfrak{C}'_2$ different from $(\eta_1,0,v_\xi)$ and let
$(\eta_j,u_j,v_j)\in \mathfrak{C}_2$ such that
$$
(\eta_j,u_j,v_j)\rightarrow (\eta,u,v)\quad\text{in}\quad \R\times \Wq\times\Wq\ .
$$
Clearly, $u= 0$ or $v= 0$. Observing that
$$
\partial_a v_j-\Delta_D v_j=-\beta_1 v_j^2+\beta_2 u_j v_j\ge-\beta_1 v_j^2\ ,\,\quad  v_j(0)=\xi V_j\ ,
$$
whence $v_j\ge v_\xi$ by Lemma~\ref{A3}, we deduce $v\ge v_\xi$ and so $u= 0$ since $(\eta,u,v)\in \mathfrak{C}'_2\setminus\mathfrak{C}_2$. 
Therefore, $(\eta,u,v)=(\eta,0,v_\xi)$ by the uniqueness statement of Theorem~\ref{A1}. A similar, but simpler argument as in Lemma~\ref{L33} (see also \cite[Lem.6.5.3]{LopezGomezChapman} or \cite[Lem.4.5]{WalkerJRAM}) then implies that $1/\eta$ is an eigenvalue of $H_{[-\alpha_2 v_\xi]}$ with a positive eigenvector, that is, $\eta=\eta_1$ by Lemma~\ref{A2} and \eqref{17} yielding the contradiction that $(\eta,u,v)$ coincides with $(\eta_1,0,v_\xi)$.
Consequently, $\mathfrak{C}_2=\mathfrak{C}'_2$ does not leave $\R^+\times\Wqd\times\Wqd$ except at $(\eta_1,0,v_\xi)$, and we conclude that alternative (ii) above is impossible. An argument similar to the proof of Lemma~\ref{L33} shows that alternative (iii) can also be ruled out.  Therefore, the only remaining possibility is that $\mathfrak{C}'_2=\mathfrak{C}_2$ is unbounded in $\R^+\times\Wqd\times\Wqd$.
\end{proof}

It remains to prove that there is no other bifurcation point on the semi-trivial branches.

\begin{cor}\label{C133}
There is no other bifurcation point on $\mathfrak{B}_2$ or $\mathfrak{B}_1$ to positive coexistence solutions than $(\eta_1,0,v_\xi)\in \mathfrak{B}_2$.
\end{cor}

\begin{proof}
If $(\eta,0,v_\xi)\in \mathfrak{B}_2$ is a bifurcation point to positive coexistence solutions, then $1=\eta r(H_{[-\alpha_2 v_\xi]})$ as in the proof of Corollary~\ref{C13} and Lemma~\ref{L33} by using \eqref{19}, whence $\eta=\eta_1$.

Suppose that there is a bifurcation point $(\eta,u_{\eta},0)$ on $\mathfrak{B}_1$ to positive coexistence solutions. Then we deduce $1=\xi r(\hat{H}_{[-\beta_2 u_\eta]})$ as in the proof of Corollary~\ref{C13} and Lemma~\ref{L33}. However, this is not possible since $\xi>1$ and $1=r(\hat{H}_{[0]})<r(\hat{H}_{[-\beta_2 u_\eta]})$ by \eqref{1mio} and Lemma~\ref{A2}.
\end{proof}

This completes the proof of Theorem~\ref{T1}.
The proof of the following characterization of $\mathfrak{C}_2$ is the same as for Corollary~\ref{L3}:

\begin{cor}\label{L3a}
The continuum $\mathfrak{C}_2$ is unbounded with respect to both the parameter $\eta$ and the $u$-component in $\Wq$, or with respect to the $v$-component in $\Wq$. If \eqref{B} holds
for some $s>0$, then $\mathfrak{C}_2$ is unbounded with respect to the $u$-component in $\Wq$.
\end{cor}

\section{Competing Systems: Proof of Theorem~\ref{T3}}\label{sec 2c}
We next consider \eqref{1comp}-\eqref{2comp} subject to \eqref{3}-\eqref{4}. The simplest case  is $\xi\le 1$ when no coexistence solutions exist:

\begin{lem}\label{L44}
If $\xi\le 1$, then there is no solution $(\eta,u,v)$ in $\R^+\times\Wq^+\times\Wqd$ to \eqref{1comp}-\eqref{2comp} subject to \eqref{3}-\eqref{4}. 
\end{lem}

\begin{proof} Let $(\eta,u,v)\in\R^+\times\Wq^+\times\Wqd$ solve \eqref{1comp}-\eqref{2comp} subject to \eqref{3}-\eqref{4}.
Since $u\ge 0$, we have
\bqn\label{pp}
\partial_a v-\Delta_D v\le -\beta_1 v^2\ \text{on}\ J\times\Om\ ,\quad v(0)=\xi V\ ,
\eqn
and thus $z'(a)\le-\lambda_1 z(a)$ for $ a\in J$, where
$$
z(a):=\int_\Om \varphi_1 v(a)\,\rd x\ ,\quad a\in J\ ,
$$
and $\varphi_1$ is a positive eigenfunction for the principal eigenvalue $\lambda_1>0$ of $-\Delta_D$.
Therefore,
$$
z(0)=\xi\int_0^{a_m} b_2(a)\int_\Om \varphi_1 v(a)\,\rd a\rd x \le\xi \int_0^{a_m} b_2(a) e^{-\lambda_1 a}\,\rd a\, z(0)\ .
$$
Actually, this inequality is strict and $z(0)>0$ due to $v\in\Wqd$ and \eqref{pp}. So $\xi>1$ by \eqref{6}.
\end{proof}

In the sequel, let $\xi>1$ be arbitrarily fixed. The remainder of the proof of Theorem~\ref{T3} is then very similar to the one of Theorem~\ref{T1}, and we give merely a brief sketch of the proof mainly pointing out the differences. Due to \eqref{6} and Lemma~\ref{A2} we have
\bqn\label{xi2}
\eta_2:=\eta_2(\xi):=\frac{1}{r(H_{[\alpha_2 v_\xi]})}\in (1,\infty)\ .
\eqn
We linearize around $(\eta_2,0,v_\xi)\in\mathfrak{B}_2$ by
writing solutions to \eqref{1comp}-\eqref{2comp} subject to \eqref{3}-\eqref{4} in the form $(\eta,u,v)=(\eta,u,v_\xi-w)\in \R\times\Wq\times \Wq$ with 
\bqn\label{19s}
(u,w)-\hat{K}(\eta)(u,w)+\hat{R}(u,w)=0\ .
\eqn
Hereby,
$$
\hat{K}(\eta)(u,w):=\left(\begin{matrix} \tilde{Z}_1(0,\eta U)\\ \hat{Z}_2(\beta_2 u v_\xi,\xi W)\end{matrix}\right)\ ,\qquad \hat{R}(u,w):=-\left(\begin{matrix} \hat{Z}_1(-\alpha_1 u^2+\alpha_2uw,0)\\ \hat{Z}_2(\beta_1w^2-\beta_2uw,0)\end{matrix}\right)\ 
$$
for $(u,w)\in \Wq\times\Wq$ with
\begin{align*}
&\hat{Z}_1:=\big(\partial_a-\Delta_D+\alpha_2 v_\xi,\gamma_0)^{-1}\in \ml(\Lq\times\Wqq, \Wq)\ ,\\
&\hat{Z}_2:=\big(\partial_a-\Delta_D+2\beta_1 v_\xi,\gamma_0)^{-1}\in \ml(\Lq\times\Wqq, \Wq)\ .
\end{align*}
Exactly as in the previous sections we deduce with the aid of \cite{CrandallRabinowitz, LopezGomezChapman}: there is a continuum $\mathfrak{C}'_3$ of solutions $(\eta,u,v)$ to \eqref{1comp}-\eqref{2comp} subject to \eqref{3}-\eqref{4} in $\R\times\Wq\times\Wq$ emanating from $(\eta_2,0,v_\xi)\in\mathfrak{B}_2$ and satisfying the alternatives
\begin{itemize}
\item[(i)] $\mathfrak{C}'_3$ is unbounded in $\R\times\Wq\times\Wq$, or
\item[(ii)]  there is $\eta\in\{\zeta\in\R ; \mathrm{dim}(\mathrm{ker}(1-\hat{K}(\zeta)))\ge 1\}$ with $\eta\not=\eta_2$ and $(\eta,0,v_\xi)\in \mathfrak{C}'_3$, or
\item[(iii)] there is $(\eta,u,v_\xi-w)\in\mathfrak{C}'_3$ with $(u,w)\in\mathrm{rg}(1-\hat{K}(\eta_2))\setminus\{(0,0)\}$.
\end{itemize}
Close to $(\eta_2,0,v_\xi)$, the continuum $\mathfrak{C}'_3$ is a continuous curve in $\R^+\times\Wqd\times\Wqd$. We define
$$
\mathfrak{C}_3:=\mathfrak{C}'_3\cap (\R^+\times\Wqd\times\Wqd)\not=\emptyset\ ,
$$
and note that $\eta>1$ for $(\eta,u,v)\in\mathfrak{C}_3$ by reproducing the proof of Lemma~\ref{L44}.

\begin{lem}\label{L55}
The continuum $\mathfrak{C}_3$ is either unbounded in $\R^+\times\Wqd\times\Wqd$ or joins $\mathfrak{B}_2$ with $\mathfrak{B}_1$.
\end{lem}

\begin{proof} 
Assume that $\mathfrak{C}_3=\mathfrak{C}'_3$, that is, $\mathfrak{C}'_3\subset\R^+\times\Wqd\times\Wqd$. Then alternative (ii) above is impossible while alternative (iii) can be ruled out with the same argument as in the proof of Lemma~\ref{L33}. Thus, if $\mathfrak{C}_3=\mathfrak{C}'_3$, then it is unbounded in $\R^+\times\Wqd\times\Wqd$. Suppose that $\mathfrak{C}_3$ is a proper subset of $\mathfrak{C}'_3$. Let $((\eta_j,u_j,v_j))_{j\in\N}$ be a sequence in $\mathfrak{C}_3$ converging toward $(\eta,u,v)\in \mathfrak{C}'_3\setminus \mathfrak{C}_3$ with $(\eta,u,v)\not=(\eta_2,0,v_\xi)$, so $u= 0$ or $v= 0$. As the only solutions close to $\mathfrak{B}_0$ lie on the curve $\mathfrak{B}_1$, the case $(u,v)=(0,0)$ is impossible. If $u= 0$ but $v\not= 0$, then $(\eta,u,v)=(\eta,0,v)$ and so, since $\xi>1$, $v=v_\xi$ by Theorem~\ref{A1}. A similar, but simpler argument as in Lemma~\ref{L33} (see also \cite[Lem.6.5.3]{LopezGomezChapman} or \cite[Lem.4.5]{WalkerJRAM}) then implies that $\eta=\eta_2$. This yields the contradiction $(\eta,u,v)=(\eta_2,0,v_\xi)$. Therefore, the only remaining possibility is that $v= 0$ but $u\not= 0$ from which  $(\eta,u,v)=(\eta,u_\eta,0)\in\mathfrak{B}_1$ according to \eqref{1comp}, \eqref{3}, and Theorem~\ref{A1}. This proves the claim.
\end{proof}

The next lemma implies, in particular, that if $\mathfrak{C}_3$ is unbounded in $\R^+\times\Wqd\times\Wqd$, then it is unbounded with respect to the parameter $\eta$:

\begin{lem}\label{L10}
Given $M>1$ there is $c(M)>0$ such that $\|u\|_{\Wq}+\|v\|_{\Wq}\le c(M)$ whenever $(\eta,u,v)\in\mathfrak{C}_3$ with $\eta\le M$.
\end{lem}

\begin{proof}
Let $(\eta,u,v)\in\mathfrak{C}_3$ with $\eta\le M$. Recall $\xi, \eta>1$ and observe 
$$
\partial_a v-\Delta_D v=-\beta_1 v^2-\beta_2 u v\le-\beta_1 v^2\ ,\,\quad  v(0)=\xi V\ ,
$$
whence
\bqn\label{ui}
0\le v(a)\le v_\xi(a)\le \kappa \xi^2\ ,\quad a\in J\ ,
\eqn
by Lemma~\ref{A3} and Theorem~\ref{A1}. Similarly,
\bqn\label{uii}
0\le u(a)\le u_\eta(a)\le \kappa \eta^2\le \kappa M^2\ ,\quad a\in J\ .
\eqn 
Hence 
$$
\|u(a)\|_\infty+\|v(a)\|_\infty \le c(M)\ ,\quad a\in J\ .
$$
and we conclude with the help of Lemma~\ref{L00}.
\end{proof}

To show that $\mathfrak{C}_3$ joins $\mathfrak{B}_2$ with $\mathfrak{B}_1$ for certain values of $\xi$, we require the following auxiliary result. Recall that $\varphi_1$ is the positive eigenfunction of $-\Delta_D$ corresponding to the principal eigenvalue $\lambda_1>0$ with $\|\varphi_1\|_\infty=1$.

\begin{lem}\label{L66}
Set $\mu_1:=\lambda_1+\alpha_2\kappa \xi^2$ and $m_0:=e^{\alpha_2\kappa \xi^2 a_m}$. Given $\eta>m_0$ let $z_\eta(a):=f_\eta(a) \varphi_1$, $a\in J$, where
\bqnn
f_\eta(a):=\frac{\mu_1}{c_\eta\mu_1e^{\mu_1 a}-\alpha_1}\ ,\quad a\in J\ ,\qquad c_\eta:=\frac{\alpha_1}{\mu_1}\frac{\eta -e^{-\lambda_1 a_m}}{\eta-m_0}\ .
\eqnn
Then $z_\eta$ is increasing in $\eta\in (m_0,\infty)$ and $z_\eta\le u$ on $J\times\Om$ for any ($\eta,u,v)\in\mathfrak{C}_3$ with $\eta>m_0$.
\end{lem}

\begin{proof}
Note that $f_\eta'+\mu_1 f_\eta=-\alpha_1 f_\eta^2$, whence
$$
\partial_a z_\eta -\Delta_D z_\eta =-\alpha_1 z_\eta^2-\alpha_2 \kappa \xi^2 z_\eta -F\ \text{on}\ J\times\Omega\ ,
$$
where $F:=\alpha_1 (f_\eta -z_\eta) z_\eta\ge 0$. Due to the definition of $c_\eta>\alpha_1/\mu_1$ and \eqref{6}, it is easily seen that $z_\eta (0)\le \eta Z_\eta$. The comparison principle stated in Lemma~\ref{A3a}, \eqref{ui}, and \eqref{1comp} then yield $u\ge z_\eta $ on $J\times\Om$ for any ($\eta,u,v)\in\mathfrak{C}_3$ with $\eta>m_0$. That $z_\eta$ is increasing in $\eta\in(m_0,\infty)$ follows from $\partial_\eta f_\eta (a)>0$ for~$a\in J$.
\end{proof}

We remark that \eqref{1mio}, Lemma~\ref{L66}, Lemma~\ref{A2}, and Theorem~\ref{A1} ensure that both maps $\eta\mapsto r(\hat{H}_{[\beta_2 z_\eta]})$ and $\eta\mapsto r(\hat{H}_{[\beta_2 u_\eta]})$ belong to $
C\big( (m_0,\infty),(0,1)\big)$ and are strictly decreasing with $r(\hat{H}_{[\beta_2 u_\eta]})\le r(\hat{H}_{[\beta_2 z_\eta]})$ for $\eta>m_0$ and $\lim_{\eta\rightarrow 0} r(\hat{H}_{[\beta_2 u_\eta]})=1$, hence
$$
\big(\lim_{\eta\rightarrow \infty} r(\hat{H}_{[\beta_2 u_\eta]})\big)^{-1}\ge \big(\lim_{\eta\rightarrow \infty} r(\hat{H}_{[\beta_2 z_\eta]})\big)^{-1}=:N\in (1,\infty]
$$
is well-defined.

\begin{rem}\label{R42}
As Lemma~\ref{L66} ensures $\|z_\eta\|_\infty\rightarrow \infty$ for $\eta\rightarrow\infty$, we conjecture $N=\infty$.
\end{rem}

Note that for any $\xi\in (1,N)$ there is a unique $\eta_3:=\eta_3(\xi)>1$ such that
\bqn\label{xi3}
\xi\, r(\hat{H}_{[\beta_2 u_{\eta_3}]}) =1\ .
\eqn
In this case, the continuum $\mathfrak{C}_3$ connects $\mathfrak{B}_2$ with $\mathfrak{B}_1$ and the value of $\eta_3$ determines the point where $\mathfrak{C}_3$ joins up with $\mathfrak{B}_1$:

\begin{cor}\label{C77}
If $\xi\in (1,N)$, then $\mathfrak{C}_3$ joins up with $\mathfrak{B}_1$ at the point $(\eta_3,u_{\eta_3},0)$.
\end{cor}

\begin{proof}
Let $\xi\in (1,N)$.
If $(\eta,u,v)\in\mathfrak{C}_3$ with $\eta>m_0$, then $u\ge z_\eta$ by Lemma~\ref{L66}, hence 
$$
1=r(\xi \hat{H}_{[\beta_1 v+\beta_2 u]})\le r(\xi \hat{H}_{[\beta_2 z_\eta]})
$$
owing to Lemma~\ref{A2} and \eqref{sp2}.
Since the right hand side tends to $\xi/N<1$ as $\eta\rightarrow\infty$, there must be some $M=M(\xi)>1$ such that \mbox{$\eta\le M$} for any $(\eta,u,v)\in\mathfrak{C}_3$. Thus $\mathfrak{C}_3$ joins up with $\mathfrak{B}_1$ due to Lemma~\ref{L55} and Lemma~\ref{L10}, say, at $(\hat{\eta},u_{\hat{\eta}},0)$.
To determine $\hat{\eta}$ we first recall that $$(\eta,u,v)=(\eta,u_\eta-w,v)\in \R\times\Wq\times \Wq$$ solves 
\eqref{1comp}-\eqref{2comp} subject to \eqref{3}-\eqref{4} if and only if $(\eta,w,v)\in \R\times\Wq\times \Wq$ solves
\begin{align}
&\partial_aw-\Delta_Dw=\alpha_1w^2-2\alpha_1 u_\eta w+\alpha_2 u_\eta v -\alpha_2 v w\ ,& w(0)=\eta W\ ,\label{18aaa}\\
&\partial_av-\Delta_Dv=-\beta_1v^2-\beta_2 v(u_\eta-w)\ ,& v(0)=\xi V\ ,\label{18bbb}
\end{align}
where we put
$$
W:=\int_0^{a_m} b_1(a)\, w(a)\,\rd a\ ,\qquad V:=\int_0^{a_m} b_2(a)\, v(a)\,\rd a\ .
$$
Introducing
$$
T:=\big(\partial_a-\Delta_D,\gamma_0)^{-1}\in \ml(\Lq\times\Wqq, \Wq)
$$
and the operators 
$$
K_*(\eta)(w,v):=\left(\begin{matrix} T(-2\alpha_1 u_\eta w+\alpha_2 u_\eta v ,\eta W)\\ T(-\beta_1 u_\eta v,\xi V)\end{matrix}\right)\ ,\qquad R_*(w,v):=-\left(\begin{matrix} T(\alpha_1 w^2-\alpha_2 vw,0)\\ T(-\beta_1 v^2+\beta_2 wv,0)\end{matrix}\right)\ 
$$ 
acting on $(w,v)\in \Wq\times\Wq$, equations \eqref{18aaa}, \eqref{18bbb} are equivalent to
\bqn\label{1999}
(w,v)-K_*(\eta)(w,v)+R_*(w,v)=0\ .
\eqn
The operators $K_*$ and $R_*$ possess the properties stated in \eqref{K}-\eqref{RR}. Now, as $\mathfrak{C}_3$ joins up with $\mathfrak{B}_1$ at $(\hat{\eta},u_{\hat{\eta}},0)$, there is a sequence $((\eta_j,u_j,v_j))_j$ in $\mathfrak{C}_3$ converging to $(\hat{\eta},u_{\hat{\eta}},0)$.
Set $w_j:=u_{\eta_j}-u_j$ and note that $w_j\in\Wq^+$ according to \eqref{uii}. As $u_\eta$ depends continuously on $\eta$, formulation \eqref{1999} and the properties of $K_*$ and $R_*$ readily imply (see, e.g., the proof of \cite[Lem.6.5.3]{LopezGomezChapman} or Lemma~\ref{L33}) that
$$
\frac{(w_j,v_j)}{\|(w_j,v_j)\|_{\Wq\times\Wq}}
$$
converges to some eigenvector $(\phi,\psi)\in \Wq^+\times\Wq^+$ of $K_*(\hat{\eta})$ associated to the eigenvalue $1$ and thus satisfying \eqref{18aaa}, \eqref{18bbb} with $\eta=\hat{\eta}$ when higher order terms are neglected:
\begin{align*}
\partial_a\phi-\Delta_D\phi &=-2\alpha_1 u_{\hat{\eta}}\phi+\alpha_2 u_{\hat{\eta}}\psi\ ,& \phi(0)=\hat{\eta}\,\Phi\ ,\\
\partial_a\psi-\Delta_D\psi &=-\beta_2 u_{\hat{\eta}}\psi\ ,& \psi(0)=\xi\Psi\ .
\end{align*}
Suppose $\psi= 0$. Then the first equation yields $\phi(0)=\hat{\eta} H_{[2\alpha_1 u_{\hat{\eta}}]} \phi(0)$ and thus, since $\phi(0)\in\Wqq\setminus\{0\}$, we obtain from Lemma~\ref{A2} and \eqref{sp} the contradiction
$$
1\le \hat{\eta} r\big(H_{[2\alpha_1 u_{\hat{\eta}}]}\big) < \hat{\eta} r\big(H_{[\alpha_1 u_{\hat{\eta}}]}\big) =1\ .
$$
Therefore, $\psi\not= 0$ and hence $\psi(0)\in\Wqq\setminus\{0\}$. The equation for $\psi$ ensures $\psi(0)=\xi \hat{H}_{[\beta_2 u_{\hat{\eta}}]} \psi(0)$, whence $\xi r(\hat{H}_{[\beta_2 u_{\hat{\eta}}]})=1$. We conclude $\hat{\eta}=\eta_3$ according to \eqref{xi3}.
\end{proof}

Finally, we show that $\mathfrak{C}_3$ connects the two semi-trivial branches if assumption \eqref{dominated} holds.

\begin{cor}\label{C90}
Suppose \eqref{dominated} and let $\xi>1$ be arbitrary. Then the $\eta$-projection of $\mathfrak{C}_3$ is contained in the interval $(1,\xi]$. In particular, $\mathfrak{C}_3$ joins $\mathfrak{B}_2$ with $\mathfrak{B}_1$.
\end{cor}

\begin{proof}
Given $\Phi\in\Wqqp$ and $u,v\in\Wq^+$ we have
\bqnn
\begin{split}
\big(H_{[\alpha_1 u+\alpha_2 v]}-\hat{H}_{[\beta_1 v+\beta_2 u]}\big)\Phi =& \int_0^{a_m} \big( b_1(a)-b_2(a)\big)\Pi_{[\beta_1 v+\beta_2 u]}(a,0)\Phi\,\rd a \\
&+ \int_0^{a_m} b_1(a) \big( \Pi_{[\alpha_1 u+\alpha_2 v]}(a,0)- \Pi_{[\beta_1 v+\beta_2 u]}(a,0) \big) \Phi\,\rd a \ .
\end{split}
\eqnn
Since $\alpha_1 u+\alpha_2 v\le \beta_1 v+\beta_2 u$ by \eqref{dominated}, the parabolic maximum principle implies 
$$
\big(\Pi_{[\alpha_1 u+\alpha_2 v]}(a,0)- \Pi_{[\beta_1 v+\beta_2 u]}(a,0) \big) \Phi\ge 0\ \,\text{on}\ \, \Om\ ,\quad a\in J\ ,
$$
whence $H_{[\alpha_1 u+\alpha_2 v]}\ge \hat{H}_{[\beta_1 v+\beta_2 u]}$ by the above equality from which 
$$
r(H_{[\alpha_1 u+\alpha_2 v]})\ge r(\hat{H}_{[\beta_1 v+\beta_2 u]})
$$ 
due to \cite[Thm.3.2(v)]{AmannSIAMReview}. Thus, given $(\eta,u,v)\in\mathfrak{C}_3$ we have
$$
1=\xi r(\hat{H}_{[\beta_1 v+\beta_2 u]})\le \xi r(H_{[\alpha_1 u+\alpha_2 v]})=\frac{\xi}{\eta}
$$
by \eqref{sp2} and \eqref{sp1}. So the $\eta$-projection of $\mathfrak{C}_3$ is contained in $(1,\xi]$. Due to Lemma~\ref{L55} and Lemma~\ref{L10}, this in particular implies that $\mathfrak{C}_3$ joins $\mathfrak{B}_2$ with $\mathfrak{B}_1$.
\end{proof}

This completes the proof of Theorem~\ref{T3}.


\appendix

\section{Auxiliary Results: Semi-Trivial Branches}\label{sec 3}

In this appendix we collect certain results regarding the parameter-dependent equation
\bqn\label{Ae1}
\partial_a u-\Delta_D u=-\alpha_1u^2\ ,\quad u(0,\cdot)=\eta\int_0^{a_m}b_1(a)\, u(a,\cdot)\,\rd a\ .
\eqn
Most of these results have been proved in \cite{WalkerJRAM}. 

Suppose \eqref{5} and \eqref{6} in the following. We first recall a comparison principle (see \cite[Lem.3.2]{WalkerJRAM}) for parabolic equations with nonlocal initial conditions of the form \ref{Ae1}, which, in particular, guarantees uniqueness of positive solutions:

\begin{lem}\label{A3}
Let $\eta>1$ and $F\in \Lq^+$. Suppose $u,v\in\Wqd$ satisfy either
\begin{align*}
&\partial_au-\Delta_D u=-\alpha_1 u^2+F\quad \text{in}\ J\times \Om\ , & u(0)\ge\eta \int_0^{a_m}b_1(a)\,u(a)\,\rd a\ ,\\
&\partial_av-\Delta_D v=-\alpha_1 v^2\quad  \text{in}\ J\times \Om\ , & v(0)=\eta \int_0^{a_m}b_1(a)\,v(a)\,\rd a\ ,
\end{align*}
or
\begin{align*}
&\partial_au-\Delta_D u=-\alpha_1 u^2\quad  \text{in}\ J\times \Om\  , & u(0)=\eta \int_0^{a_m}b_1(a)\,u(a)\,\rd a\ ,\\
&\partial_av-\Delta_D v=-\alpha_1 v^2-F\quad  \text{in}\ J\times \Om\ , & v(0)\le \eta \int_0^{a_m}b_1(a)\,v(a)\,\rd a\ .
\end{align*}
Then $u\ge v$.
\end{lem}

Along the lines of the proof of \cite[Lem.3.2]{WalkerJRAM} one may also derive the following variant:

\begin{lem}\label{A3a}
Let $\eta>1$ and $F\in \Lq^+$. Let $R>0$ and suppose $u,w,v\in\Wqd$ with $v\le R$ on $J\times\Om$ satisfy 
\begin{align*}
&\partial_au-\Delta_D u=-\alpha_1 u^2-\alpha_2uv\quad  \text{in}\ J\times \Om\ , & u(0)=\eta \int_0^{a_m}b_1(a)\,u(a)\,\rd a\ ,\\
&\partial_aw-\Delta_D w=-\alpha_1 w^2-\alpha_2 Rw-F\quad  \text{in}\ J\times \Om\ , & w(0)\le\eta \int_0^{a_m}b_1(a)\,w(a)\,\rd a\ .
\end{align*}
Then $u\ge w$.
\end{lem}

Properties of solutions to \eqref{Ae1} are connected to operators of the form $H_{[h]}$ as introduced in Section~\ref{sec 2}. The next lemma is a consequence of the famous Krein-Rutman theorem \cite[Thm.3.2]{AmannSIAMReview} and gives information about the spectral radii of such operators. We refer to \cite[Lem.3.1]{WalkerJRAM} for a proof. Actually, the proof of Lemma~\ref{A3} given in \cite{WalkerJRAM} is  based on the next lemma. 

\begin{lem}\label{A2}
For $h\in C^\varrho (J,C(\bar{\Om}))$ with $\varrho>0$, the operator $H_{[h]}\in\mk(\Wqb^{2-2/q})$ is strongly positive, i.e. \eqref{ssss} holds. In particular, the spectral radius $r(H_{[h]})>0$ is a simple eigenvalue with an eigenfunction $B_{[h]}$ belonging to $\mathrm{int}(\Wqqp)$. It is the only eigenvalue of $H_{[h]}$ with a positive eigenfunction. Moreover, if $h$ and $g$ both belong to $C^\varrho (J,C(\bar{\Om}))$ with $g\ge h$ but $g\not\equiv h$, then $r(H_{[g]})<r(H_{[h]})$.
\end{lem}

Finally, we gather results from \cite{WalkerJRAM} about properties of solutions to \eqref{Ae1} being fundamental for the investigation of \eqref{1}-\eqref{4}. Recall that $\lambda_1>0$ is the principal eigenvalue of $-\Delta_D$ with positive eigenfunction $\varphi_1$ (normalized such that $\|\varphi_1\|_\infty =1$).

\begin{thm}\label{A1}
For each $\eta>1$ there is a unique solution $u_\eta\in\Wq^+\setminus\{0\}$ to equation \eqref{Ae1}. The mapping $(\eta\mapsto u_\eta)\in C^\infty((1,\infty),\Wq)$ is real analytic with $\|u_\eta\|_{\Wq}\rightarrow 0$ as $\eta\rightarrow 1$ and  $\|u_\eta\|_{\Wq}\rightarrow\infty$ as $\eta\rightarrow\infty$. 
There is $\kappa>0$ such that, for $\eta>1$,
\bqn\label{we}
\kappa \eta^2\ge u_\eta(a)\ge\frac{\lambda_1}{\alpha_1}\frac{\eta-1}{\eta(e^{\lambda_1 a}-1)+1-e^{-\lambda_1 (a_m-a)}}\,\varphi_1\quad \text{on}\ \,\Om ,\quad a\in J\ ,
\eqn
and $\frac{\partial}{\partial\eta}u_\eta (a)\in\mathrm{int}(\Wqqp)$ for $a\in J$. If $\eta_1 > \eta_2$, then $u_{\eta_1}\ge u_{\eta_2}$. Finally, if $\eta\le 1$, then \eqref{Ae1} has no solution in $\Wq^+\setminus\{0\}$.
\end{thm}

The proof of this theorem is given in \cite[Thm.2.1, Cor.3.3, Lem.3.6, Lem.3.7]{WalkerJRAM} except for the analyticity of the mapping $\eta\mapsto u_\eta$.
However, this follows exactly as in the proof of \cite[Thm.2.1]{WalkerJRAM} (see subsection~3.3 therein) by taking into account the real analyticity of the mapping
$$
\Gamma: (1,\infty)\times \Wq\rightarrow\Lq\times\Wqq\, ,\quad (\eta,u)\mapsto\left(\partial_au-\Delta_Du+\alpha_1 u^2 \,,\,  u(0)-\eta \int_0^{a_m}b_1(a)\, u(a,\cdot)\,\rd a\right)
$$
and invoking the implicit function theorem for analytic maps, e.g. \cite[Thm.4.5.4]{BuffoniToland}.


\end{document}